\pdfoutput=1
\documentclass[12pt,reqno]{amsart}

\usepackage{geometry}
\usepackage{hhline}

\usepackage{mathtools}

\usepackage{mathdots}

\usepackage{color}

\usepackage{pdfsync}

\usepackage{enumitem}

\usepackage{wasysym}

\usepackage{amssymb}

\usepackage{mathrsfs}

\usepackage{bbm}

\DeclareMathAlphabet{\mathpzc}{OT1}{pzc}{m}{it}

\usepackage[all]{xy}

\usepackage{tikz}
\usetikzlibrary{arrows,matrix,decorations.pathmorphing,decorations.pathreplacing,positioning,shapes.geometric,shapes.misc,decorations.markings,decorations.fractals,calc,patterns}

\usepackage{graphicx}

\usepackage{float}

\usepackage[bottom]{footmisc}

\usepackage{moreenum}

\usepackage{makecell}

\usepackage{relsize}

\entrymodifiers={+!!<0pt,\fontdimen22\textfont2>}

\usepackage{scalerel,stackengine}
\stackMath
\newcommand\newcheck[1]{%
\savestack{\tmpbox}{\stretchto{%
  \scaleto{%
    \scalerel*[\widthof{\ensuremath{#1}}]{\kern-.6pt\bigwedge\kern-.6pt}%
    {\rule[-\textheight/2]{1ex}{\textheight}}
  }{\textheight}%
}{0.5ex}}%
\stackon[1pt]{#1}{\scalebox{-1}{\tmpbox}}%
}
\stackMath
\newcommand\newhat[1]{%
\savestack{\tmpbox}{\stretchto{%
  \scaleto{%
    \scalerel*[\widthof{\ensuremath{#1}}]{\kern-.6pt\bigwedge\kern-.6pt}%
    {\rule[-\textheight/2]{1ex}{\textheight}}
  }{\textheight}%
}{0.5ex}}%
\stackon[1pt]{#1}{\scalebox{1}{\tmpbox}}%
}

\def\cP{\mathscr{P}}

\def\BN{\mathbb{N}}

\def\BQ{\mathbb{Q}}

\def\BZ{\mathbb{Z}}

\mathchardef\mhyphen="2D

\def\adots{\mathinner{\mkern1mu\raise1.0pt\vbox{\kern7.0pt\hbox{.}}\mkern2mu\raise5.0pt\hbox{.}\mkern2mu\raise9.0pt\hbox{.}\mkern1mu}}

\newcommand{\cupdot}{\mathbin{\mathaccent\cdot\cup}}

\def\diag{\operatorname{diag}}
\def\dddots{\mathinner{\mkern1mu\raise10.0pt\vbox{\kern7.0pt\hbox{.}}\mkern2mu\raise5.3pt\hbox{.}\mkern2mu\raise1.0pt\hbox{.}\mkern1mu}}
\def\dddotssmall{\mathinner{\mkern1mu\raise7.0pt\vbox{\kern7.0pt\hbox{.}}\mkern-1mu\raise4pt\hbox{.}\mkern-1mu\raise1.0pt\hbox{.}\mkern1mu}}

\def\K0{\operatorname{K}_0}

\def\opp{\operatorname{op}}

\def\PSL2{\operatorname{PSL}_2}

\def\SL2{\operatorname{SL}_2}

\numberwithin{equation}{section}

\renewcommand{\labelenumi}{(\roman{enumi})}

\newtheorem{Lemma}{Lemma}[section]

\newtheorem{Proposition}[Lemma]{Proposition}

\theoremstyle{definition}
\newtheorem{Definition}[Lemma]{Definition}
\newtheorem{Setup}[Lemma]{Setup}

\newtheorem{Example}[Lemma]{Example}

\theoremstyle{theorem}
\newtheorem{ThmIntro}{Theorem}

\theoremstyle{definition}

\newtheorem*{bfhpg*}{}

\usepackage{amsthm}
\usepackage{thmtools}
\declaretheoremstyle[
notefont=\bfseries, notebraces={}{},
bodyfont=\normalfont,
headformat=\NUMBER~\NOTE,
headpunct={}
]{foobar}

  {\begin{list}{}{%
    \settowidth{\labelwidth}{\textbf{#1:}}%
    \setlength{\leftmargin}{\labelwidth}\addtolength{\leftmargin}{\labelsep}}}%
  {\end{list}}

\makeatletter
\@namedef{subjclassname@2020}{%
  \textup{2020} Mathematics Subject Classification}
\makeatother

\begin{document}

\setlength{\parindent}{0pt}
\setlength{\parskip}{7pt}

\title[Non-commutative friezes]{Non-commutative friezes and their determinants, the non-commutative Laurent phenomenon for weak friezes, and frieze gluing}

\author{Michael Cuntz}

\address{Michael Cuntz, Leibniz Universit\"at Hannover,
Institut f\"ur Algebra, Zahlentheorie und Dis\-krete Mathematik, Fakult\"at f\"ur Mathematik und Physik, Welfengarten 1, 30167 Hannover, Germany}
\email{cuntz@math.uni-hannover.de}

\urladdr{https://www.iazd.uni-hannover.de/de/cuntz}

\author{Thorsten Holm}

\address{Thorsten Holm, Leibniz Universit\"at Hannover,
Institut f\"ur Algebra, Zahlentheorie und Dis\-krete Mathematik, Fakult\"at f\"ur Mathematik und Physik, Welfengarten 1, 30167 Hannover, Germany}
\email{holm@math.uni-hannover.de}

\urladdr{https://www.iazd.uni-hannover.de/de/holm}

\author{Peter J\o rgensen}

\address{Peter J\o rgensen, Department of Mathematics, Aarhus University, Ny Munkegade 118, 8000 Aarhus C, Denmark}
\email{peter.jorgensen@math.au.dk}

\urladdr{https://sites.google.com/view/peterjorgensen}


\keywords{Cluster algebra, cluster expansion formula, Coxeter frieze, Dieudonn\'{e} determinant, generalised frieze, polygon dissection, skew field, $T$-path formula, weak frieze}

\subjclass[2020]{05E99, 13F60, 51M20}

\begin{abstract} 

This paper studies a non-commutative generalisation of Coxeter friezes due to Berenstein and Retakh.  It generalises several earlier results to this situation: A formula for frieze determinants, a $T$-path formula expressing the Laurent phenomenon, and results on gluing friezes together.  One of our tools is a non-commutative version of the weak friezes introduced by \c{C}anak\c{c}\i\ and J\o rgensen.

\end{abstract}

\maketitle

\setcounter{section}{-1}
\section{Introduction}
\label{sec:introduction}

Coxeter frieze patterns with entries in the integers were introduced in \cite[sec.\ 1]{Coxeter-frieze}; see Figure \ref{fig:frieze_pattern_integers}.  This paper studies a non-commutative generalisation due to Berenstein and Retakh \cite[sec.\ 1]{Berenstein-Retakh}; see also \cite[def.\ 3.1 and thm.\ 7.1]{Cuntz-Holm-Jorgensen-noncomm}.

Coxeter frieze patterns form a nexus between algebra, combinatorics, and geometry, and as shown by our bibliography there is a large literature on these patterns and their analogues with entries in commutative rings and their deformations; see
\cite{Assem-Dupont-Schiffler-Smith-friezes},
\cite{Assem-Reutenauer-Smith-friezes},
\cite{Baur-frieze-patterns},
\cite{Baur-Faber-Gratz-Serhiyenko-Todorov-mutation-survey},
\cite{Baur-Faber-Gratz-Serhiyenko-Todorov-SL_k},
\cite{Baur-Faber-Gratz-Serhiyenko-Todorov-mutation},
\cite{Baur-Fellner-Parsons-Tschabold},
\cite{Baur-Marsh-determinant},
\cite{Baur-Marsh-discs},
\cite{Baur-Parsons-Tschabold},
\cite{Bergeron-Reutenauer-SL_k},
\cite{Bessenrodt-walks},
\cite{Bessenrodt-Holm-Jorgensen-SL_2},
\cite{Bessenrodt-Holm-Jorgensen-friezes},
\cite{Broline-Crowe-Isaacs},
\cite{Canakci-Felikson-Garcia-Elsener-Tumarkin},
\cite{Canakci-Jorgensen},
\cite{Conway-Coxeter},
\cite{Coxeter-frieze},
\cite{Coxeter-regular-complex-polytopes},
\cite{Cuntz-combinatorial-model},
\cite{Cuntz-wild},
\cite{Cuntz-Boehmler},
\cite{Cuntz-Holm-subsets},
\cite{Cuntz-Holm-subpolygons},
\cite{Cuntz-Holm-Jorgensen-coefficients},
\cite{Cuntz-Holm-Pagano-algebraic},
\cite{Felikson-Tumarkin-surfaces},
\cite{Garcia-Elsener-Serhiyenko},
\cite{Guo-tropical},
\cite{Holm-Jorgensen-p-angulated},
\cite{Holm-Jorgensen-strip},
\cite{Holm-Jorgensen-determinant},
\cite{Holm-Jorgensen-generalised-friezes-I},
\cite{Holm-Jorgensen-generalised-friezes-II},
\cite{Morier-Genoud-counting},
\cite{Morier-Genoud-BLMS},
\cite{Morier-Genoud-Ovsienko-q-friezes},
\cite{Morier-Genoud-Ovsienko-Schwartz-Tabachnikov-linear-difference},
\cite{Morier-Genoud-Ovsienko-Tabachnikov-2-friezes},
\cite{Morier-Genoud-Ovsienko-Tabachnikov-supersymmetric},
\cite{Morier-Genoud-Ovsienko-Tabachnikov-SL2-tilings},
\cite{Propp}. 
There is a significant interface with cluster algebras, which is explained in \cite{Pressland-friezes-clusters}.

This paper goes in a different direction: We will  study non-commutative friezes and generalise se\-ve\-ral results from the commutative theory, including a formula for frieze determinants, a $T$-path formula expressing the Laurent phenomenon, and results on gluing friezes together.  One of our tools is a non-commutative version of the weak friezes introduced in \cite[def.\ 0.1(ii)]{Canakci-Jorgensen}.

We remark that Berenstein and Retakh's work has recently been generalised in a separate, cluster algebraic direction, see \cite{Greenberg-Kaufman-Niemeyer-Wienhard}.

\begin{center}
{\bf Coxeter frieze patterns}
\end{center}

Figure \ref{fig:frieze_pattern_integers} shows a Coxeter frieze pattern.  It is an infinite diagonal array with borders of zeroes and ones, and the remaining entries are positive integers such that each adjacent $2 \times 2$-matrix has determinant one.

Using the notation of Figure \ref{fig:frieze_pattern_abstract}, the entries of a portion of a Coxeter frieze pattern can be denoted $c_{ ij }$ where $1 \leqslant i,j \leqslant n$ and $i \neq j$.  The integer $n$ must be chosen to match the pattern in question; for instance, $n = 6$ for Figure \ref{fig:frieze_pattern_integers}.  We can view $c$ as a map defined on $\diag P$, the set of directed diagonals of an $n$-gon $P$.  To be precise, let the vertex set $V$ be numbered by $\{\, 1, \ldots, n \,\}$ where $1 < 2 < \cdots < n-1 < n < 1$ or $1 > 2 > \cdots > n-1 > n > 1$ in the cyclic order of the vertices.  Then
\[
  \diag P = \{\, ( i,j ) \,|\, \mbox{$i,j \in V$ and $i \neq j$} \,\},
\]
and $c$ is viewed as a map $c : \diag P \xrightarrow{} \BN$ by setting $c( i,k ) = c_{ ik }$.

It was proved by Coxeter himself in \cite[sec.\ 6]{Coxeter-frieze} that $c_{ ik } = c_{ ki }$ for each directed diagonal $( i,k )$, and it follows from \cite[thm.\ 3.3]{Cuntz-Holm-Jorgensen-coefficients} that if the directed diagonals $( i,k )$ and $( j,\ell )$ cross, then the Ptolemy relation
\begin{equation}
\label{equ:Ptolemy}
  c_{ ik }c_{ j\ell }
  = c_{ ij }c_{ \ell k } + c_{ i\ell }c_{ jk }
\end{equation}
is satisfied.  

\begin{center}
{\bf Non-commutative friezes after Berenstein and Retakh}
\end{center}

Equation \eqref{equ:Ptolemy} is part of the motivation for the following definition of friezes with values in a non-commutative ring, where Equations \eqref{equ:def:frieze:a} and \eqref{equ:def:frieze:b} are due to Berenstein and Retakh \cite[eqs.\ (1.2) and (1.3)]{Berenstein-Retakh}.

\begin{Definition}
[Friezes]
\label{def:frieze}
Let $P$ be a polygon, $R$ a ring with set of invertible elements $R^*$.  A map $c : \diag P \xrightarrow{} R^*$ is {\em a frieze} if the following are satisfied, where $c_{ ik }$ is short for $c( i,k )$.
\begin{enumerate}
\setlength\itemsep{4pt}

  \item  If $i$, $j$, $k$ are pairwise different vertices of $P$, then the {\em triangle relation}
\begin{equation}
\label{equ:def:frieze:a}	
  c_{ ij }c_{ kj }^{ -1 }c_{ ki } = c_{ ik }c_{ jk }^{ -1 }c_{ ji }  
\end{equation}
is satisfied.

  \item  If the directed diagonals $( i,k )$ and $( j,\ell )$ cross, then the {\em exchange relation}
\begin{equation}
\label{equ:def:frieze:b}
  c_{ ik } = c_{ ij }c_{ \ell j }^{ -1 }c_{ \ell k }
           + c_{ i \ell }c_{ j \ell }^{ -1 }c_{ jk }
\end{equation}
is satisfied.

\end{enumerate}
The relations are illustrated in Figure \ref{fig:def:frieze:a}.
\begin{figure}
\begin{tikzpicture}
  \matrix [column sep={0.7cm,between origins},row sep={0.7cm,between origins}]
  {
    \node{ $\dddots$ }; &&& \node{ $\dddots$ }; &&& \node{ $\dddots$ }; &&&&&&& \\
    & \node{ $0$ }; & \node{ $1$ }; & \node{ $2$ }; & \node{ $3$ }; & \node{ $1$ }; & \node{ $1$ }; & \node{ $0$ }; &&&&&& \\
    && \node{ $0$ }; & \node{ $1$ }; & \node{ $2$ }; & \node{ $1$ }; & \node{ $2$ }; & \node{ $1$ }; & \node{ $0$ }; &&&&& \\
    &&& \node{ $0$ }; & \node{ $1$ }; & \node{ $1$ }; & \node{ $3$ }; & \node{ $2$ }; & \node{ $1$ }; & \node{ $0$ }; &&&& \\
    &&&& \node{ $0$ }; & \node{ $1$ }; & \node{ $4$ }; & \node{ $3$ }; & \node{ $2$ }; & \node{ $1$ }; & \node{ $0$ }; &&& \\
    &&&&& \node{ $0$ }; & \node{ $1$ }; & \node{ $1$ }; & \node{ $1$ }; & \node{ $1$ }; & \node{ $1$ }; & \node{ $0$ }; && \\
    &&&&&& \node{ $0$ }; & \node{ $1$ }; & \node{ $2$ }; & \node{ $3$ }; & \node{ $4$ }; & \node{ $1$ }; & \node{ $0$ }; & \\
    &&&&&&& \node{ $\dddots$ }; &&& \node{ $\dddots$ }; &&& \node{ $\dddots$ }; \\
  };
\end{tikzpicture}
  \caption{A Coxeter frieze pattern.  There are borders of zeroes and ones, and the remaining entries are positive integers such that each adjacent $2 \times 2$-matrix has determinant one.}
\label{fig:frieze_pattern_integers}
\end{figure}
\begin{figure}
\begin{tikzpicture}
  \matrix [column sep={1.2cm,between origins},row sep={1.0cm,between origins}]
  {
    \node{ $0$ }; & \node{ $c_{ 12 }$ }; & \node{ $c_{ 13 }$ }; & \node{ $c_{ 14 }$ }; & \node{ $\cdots$ }; & \node{ $c_{ 1n }$ }; & \node{ $0$ }; \\
    & \node{ $0$ }; & \node{ $c_{ 23 }$ }; & \node{ $c_{ 24 }$ }; & \node{ $\cdots$ }; & \node{ $c_{ 2n }$ }; & \node{ $c_{ 21 }$ }; & \node{ $0$ }; \\
    && \node{ $0$ }; & \node{ $c_{ 34 }$ }; & \node{ $\cdots$ }; & \node{ $c_{ 3n }$ }; & \node{ $c_{ 31 }$ }; & \node{ $c_{ 32 }$ }; & \node{ $0$ }; \\
    &&& \node{ $\dddots$ }; & \node{ $\dddots$ }; & \node{ $\vdots$ }; & \node{ $\vdots$ }; & \node{ $\vdots$ }; & \node{ $\dddots$ }; & \node{ $\dddots$ }; \\
    &&&& \node{ $0$ }; & \node{ $c_{ n-1,n }$ }; & \node{ $c_{ n-1,1 }$ }; & \node{ $c_{ n-1,2 }$ }; & \node{ $\cdots$ }; & \node{ $c_{ n-1,n-2 }$ }; & \node{ $0$ }; \\
    &&&&& \node{ $0$ }; & \node{ $c_{ n,1 }$ }; & \node{ $c_{ n,2 }$ }; & \node{ $\cdots$ }; & \node{ $c_{ n,n-2 }$ }; & \node{ $c_{ n,n-1 }$ }; & \node{ $0$ }; \\
  };
\end{tikzpicture}
  \caption{Notation for the entries of a frieze pattern.}
\label{fig:frieze_pattern_abstract}
\end{figure}
\begin{figure}
\begingroup
\[
  \begin{tikzpicture}[scale=1.5]

    \begin{scope}[shift={(-2.8,0)}]
      \node[name=s, shape=regular polygon, regular polygon sides=13, minimum size=7.5cm, draw] {}; 
      \draw (90+4*360/13:2.75cm) node {$i$};
      \draw (90-5*360/13:2.75cm) node {$j$};
      \draw (90-1*360/13:2.85cm) node {$k$};
      \draw [red] (90+6*360/13:1.75cm) node {$c_{ ij }$};
      \draw [red] (90-3*360/13:1.85cm) node {$c_{ kj }^{ -1 }$};
      \draw [red] (90+1.5*360/13:1.25cm) node {$c_{ ki }$};
      \draw [blue] (90+6*360/13:1.0cm) node {$c_{ ji }$};
      \draw [blue] (90-3*360/13:1.0cm) node {$c_{ jk }^{ -1 }$};
      \draw [blue] (90+1.5*360/13:0.4cm) node {$c_{ ik }$};
      \begin{scope}[red,decoration={
    markings,mark=at position 0.45 with {\arrow[scale=2]{>}}}] 
        \draw[postaction={decorate}] (90+4*360/13:2.6cm)--(90-5*360/13:2.6cm);
        \draw[postaction={decorate}] (90-5*360/13:2.6cm)--(90-1*360/13:2.6cm);
        \draw[postaction={decorate}] (90-1*360/13:2.6cm)--(90+4*360/13:2.6cm);
      \end{scope}
      \begin{scope}[blue,decoration={
    markings,mark=at position 0.45 with {\arrow[scale=2]{>}}}] 
        \draw[postaction={decorate}] (90+4*360/13+1:2.35cm)--(90-1*360/13-1:2.35cm);
        \draw[postaction={decorate}] (90-1*360/13-1:2.35cm)--(90-5*360/13:2.35cm);
        \draw[postaction={decorate}] (90-5*360/13:2.35cm)--(90+4*360/13+1:2.35cm);
      \end{scope}
    \end{scope}

    \begin{scope}[shift={(2.8,0)}]
      \node[name=s, shape=regular polygon, regular polygon sides=13, minimum size=7.5cm, draw] {}; 
      \draw (90+4*360/13:2.75cm) node {$i$};
      \draw (90-5*360/13:2.80cm) node {$j$};
      \draw (90-1*360/13:2.85cm) node {$k$};
      \draw (90+1*360/13:2.85cm) node {$\ell$};

      \draw (180:0.95cm) node {$c_{ ik }$};
      
      \draw [red] (90-7*360/13:1.6cm) node {$c_{ ij }$};
      \draw [red] (90-2*360/13:0.7cm) node {$c_{ \ell j }^{ -1 }$};
      \draw [red] (90+0*360/13:2.0cm) node {$c_{ \ell k }$};
      \draw [blue] (90+2.5*360/13:1.65cm) node {$c_{ i\ell }$};
      \draw [blue] (90+4.5*360/13:0.15cm) node {$c_{ j\ell }^{ -1 }$};
      \draw [blue] (90-3*360/13:1.7cm) node {$c_{ jk }$};
      
      \draw[dotted,thick] (90+4*360/13:2.5cm) to (90-1*360/13:2.5cm);    
      \begin{scope}[red,decoration={
    markings,mark=at position 0.45 with {\arrow[scale=2]{>}}}] 
        \draw[postaction={decorate}] (90+4*360/13:2.5cm)--(90-5*360/13+1.5:2.5cm);
        \draw[postaction={decorate}] (90-5*360/13+1.5:2.5cm)--(90+1*360/13-1.5:2.5cm);
        \draw[postaction={decorate}] (90+1*360/13-1.5:2.5cm)--(90-1*360/13:2.5cm);
      \end{scope}
      \begin{scope}[blue,decoration={
    markings,mark=at position 0.45 with {\arrow[scale=2]{>}}}] 
        \draw[postaction={decorate}] (90+4*360/13:2.5cm)--(90+1*360/13+1.5:2.5cm);
        \draw[postaction={decorate}] (90+1*360/13+1.5:2.5cm)--(90-5*360/13-1.5:2.5cm);
        \draw[postaction={decorate}] (90-5*360/13-1.5:2.5cm)--(90-1*360/13:2.5cm);
      \end{scope}
    \end{scope}

  \end{tikzpicture} 
\]
\endgroup
\caption{The friezes of Definition \ref{def:frieze} satisfy the triangle relation $c_{ ij }c_{ kj }^{ -1 }c_{ ki } = c_{ ik }c_{ jk }^{ -1 }c_{ ji }$ (left) and the exchange relation $c_{ ik } = c_{ ij }c_{ \ell j }^{ -1 }c_{ \ell k } + c_{ i \ell }c_{ j \ell }^{ -1 }c_{ jk }$ (right).}
\label{fig:def:frieze:a}
\end{figure}
\end{Definition}

Friezes should perhaps be called non-commutative friezes, but the prefix will be omitted.  They were defined differently in \cite[def.\ 3.1]{Cuntz-Holm-Jorgensen-noncomm}, using a small subset of the equations in Definition \ref{def:frieze}, but the two definitions are equivalent by \cite[thm.\ 7.1]{Cuntz-Holm-Jorgensen-noncomm}.  If $c$ is a frieze, then Figure \ref{fig:frieze_pattern_abstract} will be called a frieze pattern.

It is a remarkable insight of Berenstein and Retakh that Equations \eqref{equ:def:frieze:a} and \eqref{equ:def:frieze:b} support a rich theory.  In particular, it is unobvious to include the triangle relation, but we shall see that it is an ingredient permitting the generalisation of several results from the commutative theory.

If $c_{ ik } = c_{ ki }$ for each directed diagonal $( i,k )$ and $R$ is commutative, then Equation \eqref{equ:def:frieze:a} is automatically satisfied, and Equation \eqref{equ:def:frieze:b} is equivalent to Equation \eqref{equ:Ptolemy}.  Hence Coxeter frieze patterns, viewed as having values in $\BQ$, provide examples of Definition \eqref{def:frieze}, but there are other, truly non-commutative examples.  The universal instance where \eqref{equ:def:frieze:a} and \eqref{equ:def:frieze:b} are viewed as relations in a free algebra is investigated in depth in \cite[sec.\ 2]{Berenstein-Retakh}.  A more concrete example with $R$ equal to the quaternions was given in \cite[sec.\ 1]{Cuntz-Holm-Jorgensen-noncomm}; it is reproduced in Figure \ref{fig:frieze_pattern_quaternions}.  In this example, $c_{ ik } \neq c_{ ki }$.
\begin{figure}
\begin{tikzpicture}
  \matrix [column sep={1.5cm,between origins},row sep={1.0cm,between origins}]
  {
     \node{ $0$ }; & \node{ $1$ }; & \node{ $i$ }; & \node{ $1-k$ }; & \node{ $-i-2j$ }; & \node{ $1$ }; & \node{ $0$ }; &&&&&& \\
    & \node{ $0$ }; & \node{ $1$ }; & \node{ $-2i-j$ }; & \node{ $3k$ }; & \node{ $-i+j$ }; & \node{ $1$ }; & \node{ $0$ }; &&&&& \\
    && \node{ $0$ }; & \node{ $1$ }; & \node{ $i-j$ }; & \node{ $k$ }; & \node{ $i$ }; & \node{ $1$ }; & \node{ $0$ }; &&&& \\
    &&& \node{ $0$ }; & \node{ $1$ }; & \node{ $j$ }; & \node{ $1+k$ }; & \node{ $-2i-j$ }; & \node{ $1$ }; & \node{ $0$ }; &&& \\
    &&&& \node{ $0$ }; & \node{ $1$ }; & \node{ $-i-2j$ }; & \node{ $-3k$ }; & \node{ $i-j$ }; & \node{ $1$ }; & \node{ $0$ }; && \\
    &&&&& \node{ $0$ }; & \node{ $1$ }; & \node{ $-i+j$ }; & \node{ $-k$ }; & \node{ $j$ }; & \node{ $1$ }; & \node{ $0$ }; & \\
  };
\end{tikzpicture}
  \caption{A frieze pattern with values in the quaternions.  Reproduced from \cite[sec.\ 1]{Cuntz-Holm-Jorgensen-noncomm}.}
\label{fig:frieze_pattern_quaternions}
\end{figure}

\begin{center}
{\bf Non-commutative frieze determinants}
\end{center}

Frieze determinants were introduced and computed by Broline, Crowe, and Isaacs in \cite[thm.\ 4]{Broline-Crowe-Isaacs}, and a beautiful generalisation was established by Baur and Marsh in \cite[thm.\ 1.3]{Baur-Marsh-determinant}.  We offer a non-commutative generalisation based on the notion of Dieudonn\'{e} determinant for which we refer the reader to \cite[sec.\ IV.1]{Artin-book}, \cite{Dieudonne}.

Our generalisation shows that, as in the commutative case, the determinant of the matrix $M_c$ in Figure \ref{fig:frieze_matrix} is independent of the bulk of the frieze $c : \diag P \xrightarrow{} R^*$, depending only on the size of $P$ and the values $c_{ 12 }$, $c_{ 23 }$, $\cdots$,  $c_{ n-1,n }$, $c_{ n1 }$ on the edges of $P$.  


\begin{ThmIntro}
[Non-commutative frieze determinants]
\label{thm:A}
Suppose that $R$ is a skew field.  Let the following be given.
\begin{itemize}
\setlength\itemsep{4pt}

  \item  $P$ is an $n$-gon for some integer $n \geqslant 3$ with vertex set numbered by $\{\, 1, \ldots, n \,\}$ such that $1 < 2 < \cdots < n-1 < n < 1$ or $1 > 2 > \cdots > n-1 > n > 1$ in the cyclic order of the vertices.

  \item  $c : \diag P \xrightarrow{} R^*$ is a frieze.

\end{itemize}
The Dieudonn\'{e} determinant of the matrix $M_c$ in Figure \ref{fig:frieze_matrix} is 
\[
  \det M_c =
  \overline{ -( -2 )^{ n-2 } }
  \cdot 
  \overline{ c_{ 12 }c_{ 23 } \cdots c_{ n-1,n }c_{ n1 } }.
\]
\begin{figure}
\[
  M_c
  =
  \left(
  \begin{array}{cccccc}
      0        & c_{ 12 } & c_{ 13 } & \cdots & c_{ 1,n-1 } & c_{ 1n } \\[2mm]
      c_{ 21 } & 0        & c_{ 23 } & \cdots & c_{ 2,n-1 } & c_{ 2n } \\[2mm]
      c_{ 31 } & c_{ 32 } & 0        & \cdots & c_{ 3,n-1 } & c_{ 3n } \\[2mm]
      \vdots   & \vdots   & \vdots   & & & \vdots \\[2mm]
      c_{ n-1,1 } & c_{ n-1,2 } & c_{ n-1,3 } & \cdots & 0 & c_{ n-1,n } \\[2mm]
      c_{ n1 } & c_{ n2 } & c_{ n3 } & \cdots & c_{ n,n-1 } & 0 \\
  \end{array}
  \right).
\]
  \caption{The square matrix $M_c$ arising from a map $c : \diag P \xrightarrow{} R$ where $P$ is an $n$-gon with vertex set numbered by $\{\, 1, \ldots, n \,\}$.  Theorem \ref{thm:A} computes $\det M_c$ when $c$ is a frieze, $R$ a skew field.}
\label{fig:frieze_matrix}
\end{figure}

\end{ThmIntro}

We use the Dieudonn\'{e} determinant because $R$ is a skew field; it has values in $( R^*/[ R^*,R^* ] ) \cupdot \{\, 0 \,\}$, the abelianisation of the multiplicative group of $R$ with zero adjoined.  If $c$ is in $R^*$ then $\overline{ c }$ is the coset of $c$ in $R^*/[ R^*,R^* ]$, and if $c$ is $0$ then $\overline{c}$ is also $0$.  If $R$ is a field, then $( R^*/[ R^*,R^* ] ) \cupdot \{\, 0 \,\}$ is $R$ and the Dieudonn\'{e} determinant specialises to the usual determinant; see \cite[sec.\ 7]{Dieudonne}.  Note that, while $M_c$ depends on how the vertices of $P$ are numbered by $\{\, 1, \ldots, n \,\}$, its determinant $\det M_c$ does not: If the numbering of the vertices is subjected to a permutation $\sigma$, then $M_c$ changes by the application of $\sigma$ to its rows and to its columns, and this multiplies $\det M_c$ by $\overline{ s^2 } = \overline{ 1 }$ where $s$ is the sign of $\sigma$; see \cite[pp.\ 153--154, parts c) and j)]{Artin-book}.

Theorem \ref{thm:A} has \cite[thm.\ 4]{Broline-Crowe-Isaacs} and \cite[thm.\ 1.3]{Baur-Marsh-determinant} as special cases, because the friezes of those two papers can be viewed as having values in the relevant fields of fractions.

Note that the proof of Theorem \ref{thm:A} relies on Proposition \ref{pro:A}, which is a non-commutative generalisation of \cite[thm.\ 3.3]{Holm-Jorgensen-determinant} on weak frieze determinants.

\begin{center}
{\bf The non-commutative Laurent phenomenon}
\end{center}

The Laurent phenomenon was discovered by Fomin and Zelevinsky in their work on cluster algebras \cite[sec.\ 3]{Fomin-Zelevinsky-cluster-algebras-I}, and a non-commutative Laurent phenomenon for friezes was established by Berenstein and Retakh in \cite[thm.\ 2.10]{Berenstein-Retakh}.  We offer a generalisation based on the notion of $T$-path in the sense of \cite[def.\ 0.2]{Canakci-Jorgensen}, which is recalled in Definition \ref{def:T-path}.

Our generalisation shows that the Laurent phenomenon is independent of the triangle relation \eqref{equ:def:frieze:a}.

We first need the following definition, which is a non-commutative version of \cite[def.\ 0.1(ii)]{Canakci-Jorgensen}.

\begin{Definition}
[Weak friezes]
\label{def:weak_frieze}
Let $P$ be a polygon, $D$ a dissection of $P$, and $R$ a ring.  A map $c : \diag P \xrightarrow{} R$ is {\em a weak frieze with respect to $D$} if the following are satisfied.
\begin{enumerate}
\setlength\itemsep{4pt}

  \item  If $( t,v )$ is in $D$ then $c_{ tv }$ is in $R^*$.
  
  \item  If the directed diagonals $( i,k )$ and $( t,v )$ cross and $( t,v )$, $( v,t )$ are in $D$, then the exchange relation
\[
  c_{ ik } = c_{ it }c_{ vt }^{ -1 }c_{ vk }
           + c_{ iv }c_{ tv }^{ -1 }c_{ tk }
\]
is satisfied.
  
\end{enumerate}
\end{Definition}

Observe that a frieze is clearly a weak frieze with respect to any dissection.  Indeed, the adjective ``weak'' has been chosen because weak friezes are required to satisfy fewer exchange relations than friezes and no triangle relations at all.  Our generalisation of \cite[thm.\ 2.10]{Berenstein-Retakh} is the following.

\begin{ThmIntro}
[The non-commutative Laurent phenomenon]
\label{thm:B}
Let $P$ be a polygon, $D$ a dissection of $P$, and $c : \diag P \xrightarrow{} R$ a map where $R$ is a ring.

Statements (i) and (ii) below are equivalent.
\begin{enumerate}
\setlength\itemsep{4pt}

  \item  $c$ is a weak frieze with respect to $D$.

  \item  $c$ satisfies the $T$-path formula with respect to $D$; that is, if $i$ and $k$ are different vertices of $P$, then 
\[
  c_{ ik } = \sum_{ p \in \cP( D,i,k ) } c_p.
\]
$c$ also satisfies that if $( t,v )$ is in $D$, then $c_{ tv }$ is in $R^*$ (this is required for the formula to make sense).

\end{enumerate}
\end{ThmIntro}

We refer to Definitions \ref{def:T-path}, \ref{def:T-path_sets}, \ref{def:T-path_formula} for the notion of $T$-paths and the notation in part (ii) of the theorem and to Example \ref{exa:T-paths} for an illustration.  The point is that the $T$-path formula expresses a Laurent phenomenon because $c_p$ is a monomial in the $c_{ j\ell }$ and their inverses.

Theorem \ref{thm:B} has \cite[thm.\ 2.10]{Berenstein-Retakh} as a special case:  Assume that $c$ is a frieze, $D$ a triangulation.  Then $c$ is a weak frieze with respect to $D$, and it is easy to see that the $T$-paths with respect to $D$ are precisely the $D$-admissible sequences in the sense of \cite[def.\ 2.9]{Berenstein-Retakh}.  Hence the sum in Theorem \ref{thm:B}(ii) over the set $\cP( D,i,k )$ of $T$-paths becomes the sum in \cite[thm.\ 2.10]{Berenstein-Retakh} over the set $Adm_D( i,k )$ of admissible sequences.  

Note that Theorem \ref{thm:B} is a non-commutative version of \cite[thm.\ A]{Canakci-Jorgensen}.

\begin{center}
{\bf Gluing friezes and weak friezes}
\end{center}

The methods we will develop for the proof of Theorem \ref{thm:B} also permit us to prove the following two theorems on gluing friezes and weak friezes together.  

\begin{ThmIntro}
[Gluing weak friezes]
\label{thm:C}
Let the following be given.
\begin{itemize}
\setlength\itemsep{4pt}

  \item  $P$ is a polygon divided into subpolygons $P^1, \ldots, P^{ \kappa+1 }$ by the dissection
\begin{equation}
\label{equ:thm:C:1}
  \{\, ( t_1,v_1 ),( v_1,t_1 ), \ldots,
       ( t_{ \kappa },v_{ \kappa } ),( v_{ \kappa },t_{ \kappa } ) \,\}
\end{equation}
for some $\kappa \geqslant 0$.  See Figure \ref{fig:thm:C:a}.
\begin{figure}
\begingroup
\[
  \begin{tikzpicture}[scale=1]
      \node[name=s, shape=regular polygon, regular polygon sides=20, rotate=9, minimum size=8cm, draw] {}; 
      \draw [thick] (10*360/20:4cm) to (15*360/20:4cm);
      \draw [thick] (8*360/20:4cm) to (16*360/20:4cm);
      \draw [thick] (3*360/20:4cm) to (18*360/20:4cm);
      \draw (10*360/20:4.3cm) node { $t_1$ };
      \draw (15*360/20:4.3cm) node { $v_1$ };
      \draw (8*360/20:4.3cm) node { $t_2$ };
      \draw (16*360/20:4.3cm) node { $v_2$ };
      \draw (8*360/20-7:4.25cm) node { $\cdot$ };
      \draw (8*360/20-11:4.25cm) node { $\cdot$ };
      \draw (8*360/20-15:4.25cm) node { $\cdot$ };
      \draw (16*360/20+7:4.25cm) node { $\cdot$ };
      \draw (16*360/20+11:4.25cm) node { $\cdot$ };
      \draw (16*360/20+15:4.25cm) node { $\cdot$ };
      \draw (3*360/20:4.35cm) node { $t_{ \kappa }$ };
      \draw (18*360/20:4.35cm) node { $v_{ \kappa }$ };
      \draw (12.5*360/20:3.35cm) node { $P^1$ };
      \draw (12*360/20:2.05cm) node { $P^2$ };
      \draw (0.5*360/20:3.4cm) node { $P^{ \kappa+1 }$ };
      \draw (0.1cm,-0.3cm) node { $\cdot$ };
      \draw (-0.2cm,-0.45cm) node { $\cdot$ };
      \draw (-0.5cm,-0.6cm) node { $\cdot$ };      
  \end{tikzpicture} 
\]
\endgroup
\caption{Theorems \ref{thm:C} and \ref{thm:D} concern a polygon $P$ divided into subpolygons $P^1, \ldots, P^{ \kappa+1 }$ by the dissection $\{\, ( t_1,v_1 ),( v_1,t_1 ), \ldots, ( t_{ \kappa },v_{ \kappa } ), ( v_{ \kappa },t_{ \kappa } ) \,\}$.}
\label{fig:thm:C:a}
\end{figure}

  \item  $P^{ \alpha }$ has a dissection  $D^{ \alpha }$ for each $\alpha$ in $\{\, 1, \ldots, \kappa+1 \,\}$.
  
  \item  $c^{ \alpha } : \diag P^{ \alpha } \xrightarrow{} R$ is a weak frieze with respect to $D^{ \alpha }$ for each $\alpha$ in $\{\, 1, \ldots, \kappa+1 \,\}$.
  
\end{itemize}
Assume that the $c^{ \alpha }$ have values on the directed diagonals in the set \eqref{equ:thm:C:1} which are invertible and agree.  That is:
\begin{itemize}
\setlength\itemsep{4pt}

  \item  $c^{ \alpha }_{ rs }$ is in $R^*$ if $( r,s )$ is in the set \eqref{equ:thm:C:1} and in $\diag P^{ \alpha }$.

  \item  $c^{ \alpha }_{ rs } = c^{ \beta }_{ rs }$ if $( r,s )$ is in the set \eqref{equ:thm:C:1} and in the intersection of $\diag P^{ \alpha }$ and $\diag P^{ \beta }$.

\end{itemize}
Then there exists a unique map $c : \diag P \xrightarrow{} R$ satisfying the following conditions.
\begin{enumerate}
\setlength\itemsep{4pt}

  \item  $c \Big|_{ \diag P^{ \alpha } } = c^{ \alpha }$ for each $\alpha$ in $\{\, 1, \ldots, \kappa+1 \,\}$.

  \item  $c$ is a weak frieze  with respect to the dissection
\[  
  D = \{\, ( t_1,v_1 ),( v_1,t_1 ),\ldots, ( t_{ \kappa },v_{ \kappa } ),( v_{ \kappa },t_{ \kappa } ) \,\} \cupdot D^1 \cupdot \cdots \cupdot D^{ \kappa+1 }.
\]

\end{enumerate}
\end{ThmIntro}

\begin{ThmIntro}
[Gluing friezes]
\label{thm:D}
Let the following be given.
\begin{itemize}
\setlength\itemsep{4pt}

  \item  $P$ is a polygon divided into subpolygons $P^1, \ldots, P^{ \kappa+1 }$ by the dissection
\begin{equation}
\label{equ:thm:D:1}
  \{\, ( t_1,v_1 ),( v_1,t_1 ), \ldots,
       ( t_{ \kappa },v_{ \kappa } ),( v_{ \kappa },t_{ \kappa } ) \,\}
\end{equation}
for some $\kappa \geqslant 0$.  See Figure \ref{fig:thm:C:a}.
  
  \item  $c^{ \alpha } : \diag P^{ \alpha } \xrightarrow{} R^*$ is a frieze for each $\alpha$ in $\{\, 1, \ldots, \kappa+1 \,\}$.
  
\end{itemize}
Assume that the $c^{ \alpha }$ have values on the directed diagonals in the set \eqref{equ:thm:D:1} which agree.  That is:
\begin{itemize}
\setlength\itemsep{4pt}

  \item  $c^{ \alpha }_{ tv } = c^{ \beta }_{ tv }$ if $( t,v )$ is in the set \eqref{equ:thm:D:1} and in the intersection of $\diag P^{ \alpha }$ and $\diag P^{ \beta }$.

\end{itemize}
By Theorem \ref{thm:C}, applied with each $D^{ \alpha }$ equal to $\emptyset$, there exists a unique map $c : \diag P \xrightarrow{} R$ satisfying the following conditions.
\begin{enumerate}
\setlength\itemsep{4pt}

  \item  $c \Big|_{ \diag P^{ \alpha } } = c^{ \alpha }$ for each $\alpha$ in $\{\, 1, \ldots, \kappa+1 \,\}$.

  \item  $c$ is a weak frieze  with respect to the dissection
\[  
  \{\, ( t_1,v_1 ),( v_1,t_1 ),\ldots, ( t_{ \kappa },v_{ \kappa } ),( v_{ \kappa },t_{ \kappa } ) \,\}.
\]

\end{enumerate}
If $c$ has values in $R^*$ then it is a frieze.
\end{ThmIntro}

Note that Theorems \ref{thm:C} and \ref{thm:D} are non-commutative versions of \cite[thms.\ B and C]{Canakci-Jorgensen}.

The paper is organised into Section \ref{sec:definitions}, which provides definitions, Section \ref{sec:preliminary}, which proves some preliminary results, and Section \ref{sec:proofs}, which proves Theorems \ref{thm:A}, \ref{thm:B}, \ref{thm:C}, and \ref{thm:D}.

\section{Definitions}
\label{sec:definitions}

\begin{Setup}
\label{set:ring}
Throughout the paper, $R$ is a ring and $R^*$ is its set of invertible elements.
\end{Setup}

\begin{Definition}
[Polygons]
\label{def:background}
{$\;$}
\begin{enumerate}
\setlength\itemsep{4pt}

  \item  A {\em polygon} $P$ is a finite set $V$ of three or more {\em vertices} equipped with a cyclic ordering.  We frequently think of $P$ geometrically realised as a convex polygon in the Euclidean plane.  The number of vertices is written $|V|$.

  \item  A {\em subpolygon} $P'$ of $P$ is a subset $V'$ of $V$ of three or more vertices equipped with the induced cyclic ordering.

  \item  A {\em diagonal} of $P$ is an ordered pair $( i,k )$ of two different vertices in $V$.  The set of di\-a\-go\-nals of $P$ is denoted $\diag P$.  Note that the prefix ``directed'', used before ``diagonal'' in the introduction, will henceforth be omitted.

  \item  A diagonal $( i,k )$ is an {\em edge} if $i$ and $k$ are neighbours in $V$.  Otherwise, it is an {\em internal diagonal}.

  \item  The diagonals $( i,k )$ and $( j,\ell )$ of $P$ {\em cross} if $i$, $j$, $k$, $\ell$ are four pairwise different vertices in $V$ such that $i < j < k < \ell$ or $i < \ell < k < j$.  Note that $( i,k )$ does not cross itself or $( k,i )$ and that an edge crosses no diagonals at all.  

  \item  A {\em dissection} of $P$ is a set $D$ of internal diagonals satisfying that if $( i,k )$ is in $D$ then $( k,i )$ is in $D$, and that if the diagonals $( i,k )$ and $( j,\ell )$ are in $D$ then they do not cross.  The number of diagonals in $D$ is written $|D|$.
  
  \item  A dissection $D$ of $P$ defines subpolygons $P^1$, $P^2$, $\ldots$ whose edges are either diagonals in $D$ or edges of $P$; see Figure \ref{fig:thm:C:a}.

  \item  If $c : \diag P \xrightarrow{} R$ is a map, then we write $c_{ ik }$ for $c( i,k )$ and use the convention that $c_{ xx } = 0$.

\end{enumerate}
\end{Definition}

$T$-paths were originally defined by Schiffler with respect to triangulations; see \cite[def.\ 1]{Schiffler-expansion}.  More general versions appeared in the work by Gunawan, Musiker, Schiffler, Thomas, and Williams on cluster expansion formulae; see \cite[sec.\ 3]{Gunawan-Musiker}, \cite[sec.\ 4]{Musiker-Schiffler-Williams}, \cite[sec.\ 3.1]{Schiffler-unpunctured-II}, \cite[sec.\ 3.1]{Schiffler-Thomas-unpunctured}.  The following generalisation to dissections was introduced in \cite[def.\ 0.2]{Canakci-Jorgensen}.

\begin{Definition}
[$T$-paths]
\label{def:T-path}
Let $P$ be a polygon, $D$ a dissection of $P$.  A sequence $( p_1, \ldots, p_{ \pi } )$ of vertices of $P$ is a {\em $T$-path from $p_1$ to $p_{ \pi }$ with respect to $D$} if the following hold.
\begin{enumerate}
\setlength\itemsep{4pt}

  \item  $p_1$ and $p_{ \pi }$ are different vertices.
  
  \item  The sets $\{\, p_1,p_2 \,\}$, $\{\, p_2,p_3 \,\}$, $\ldots$, $\{\, p_{ \pi-1 },p_{ \pi } \,\}$ are pairwise different and each of them has two elements.
  
  \item  None of the diagonals $( p_{ \alpha },p_{ \alpha+1 } )$ crosses a diagonal in $D$.
  
  \item  Each of the diagonals $( p_{ 2\alpha },p_{ 2\alpha + 1 } )$ is in $D$ and crosses the diagonal $( p_1,p_{ \pi } )$, and the crossing points progress monotonically in the direction from $p_1$ to $p_{ \pi }$.  

\end{enumerate}
Note that if $( p_1, \ldots, p_{ \pi } )$ is a $T$-path then $\pi$ is even.  Indeed, if $\pi$ were odd, then $( p_{ \pi-1 },p_{ \pi } )$ would have the form $( p_{ 2\alpha },p_{ 2\alpha+1 } )$ for some $\alpha$, and then condition (iv) would imply the false statement that $( p_{ \pi-1 },p_{ \pi } )$ crossed $( p_1,p_{ \pi } )$.

See Example \ref{exa:T-paths} for an illustration.
\end{Definition}

\begin{Definition}
[The notation $\cP$]
\label{def:T-path_sets}
Let $P$ be a polygon, $D$ a dissection of $P$.
\begin{enumerate}
\setlength\itemsep{4pt}

  \item  If $i$ and $k$ are different vertices of $P$ then
\[
  \cP( D,i,k )
  =
  \Bigg\{\!\!
    \begin{array}{l}
      \mbox{ $( p_1, \ldots, p_{ \pi } )$ is a $T$-path } \\[1mm]
      \mbox{ with respect to $D$ }
    \end{array}
  \!\!\Bigg|
    \mbox{ $p_1 = i$ and $p_{ \pi } = k$ }
  \Bigg\}.
\]
  
  \item  If $i$ and $k$ are different vertices of $P$ and $( j_1, \ldots, j_{ \rho } )$ is a list of vertices of $P$ then
\[
  \cP( D,i,k )^{ ( j_1, \ldots, j_{ \rho } ) }
  = 
  \{\, ( p_1, \ldots, p_{ \pi } ) \in \cP( D,i,k ) 
  \,|\, ( p_1, \ldots, p_{ \rho } ) = ( j_1, \ldots, j_{ \rho } ) \,\}.
\]

  \item  The notation is extended by permitting the sign ``$\neq$''; for example,
\[
  \cP( D,i,k )^{ ( i,j,\neq \ell ) }
  =
  \{\, ( p_1, \ldots, p_{ \pi } ) \in \cP( D,i,k ) 
  \,|\, p_1 = i, p_2 = j, p_3 \neq \ell \,\}.
\]

\end{enumerate}
\end{Definition}

\begin{Definition}
[The $T$-path formula]
\label{def:T-path_formula}
Let $P$ be a polygon, $D$ a dissection of $P$.  Let $c : \diag P \xrightarrow{} R$ be a map such that $c_{ tv }$ is in $R^*$ if $( t,v )$ is in $D$.

\begin{enumerate}
\setlength\itemsep{4pt}

  \item  If $p = ( p_1, \ldots, p_{ \pi } )$ is a $T$-path with respect to $D$, then we set
\[
  c_p = c_{ p_1p_2 }c_{ p_3p_2 }^{ -1 }c_{ p_3p_4 } \cdots c_{ p_{ \pi-1 }p_{ \pi-2 } }^{ -1 }c_{ p_{ \pi-1 }p_{ \pi } }.
\]
Note that the inverse elements make sense.  Indeed, the diagonals $( p_3,p_2 )$, $( p_5,p_4 )$, $\ldots$, $( p_{ \pi-1 },p_{ \pi-2 } )$ are in $D$ by Definition \ref{def:background}(vi) because the diagonals $( p_2,p_3 )$, $( p_4,p_5 )$, $\ldots$, $( p_{ \pi-2 },p_{ \pi-1 } )$ are in $D$ by Definition \ref{def:T-path}(iv).
  
  \item  Let $i$ and $k$ be different vertices of $P$.  We say that {\em $c$ satisfies the $T$-path formula from $i$ to $k$ with respect to $D$} if
\begin{equation}
\label{equ:def:T-path_formula:a}
  c_{ ik } = \sum_{ p \in \cP( D,i,k ) } c_p.
\end{equation}
If this holds for each pair of different vertices, then we say that {\em $c$ satisfies the $T$-path formula with respect to $D$}.

\end{enumerate}
\end{Definition}

\begin{Example}
\label{exa:T-paths}
In Figure \ref{fig:exa:T-paths:a}, the top line shows a polygon $P$, a (red) polygon dissection
\[
  D = \{\, ( s,u ), ( u,s ), ( t,v ), ( v,t ) \,\},
\]  
and a diagonal $( i,k )$.  The bottom line shows the $T$-paths from $i$ to $k$ with respect to $D$.  
\begin{figure}
\begingroup
\[
  \begin{tikzpicture}[auto]

    \begin{scope}[shift={(0,4.1)}]
      \node[name=s, shape=regular polygon, regular polygon sides=13, minimum size=3cm, draw] {}; 
      \draw[dotted,thick] (90+5*360/13:1.5cm) to (90-1*360/13:1.5cm);    
      \draw (90+5*360/13:1.75cm) node {$i$};
      \draw (90-1*360/13:1.8cm) node {$k$};
      \draw (90-0*360/13:1.75cm) node {$t$};
      \draw (90-5*360/13:1.75cm) node {$v$};
      \draw (90+2*360/13:1.75cm) node {$s$};
      \draw (90-6*360/13:1.75cm) node {$u$};
      \draw[thick,red] (90-0*360/13:1.5cm) to (90-5*360/13:1.5cm);
      \draw[thick,red] (90+2*360/13:1.5cm) to (90-6*360/13:1.5cm);
    \end{scope}

    \begin{scope}[shift={(-5.7,0)}]
      \node[name=s, shape=regular polygon, regular polygon sides=13, minimum size=3cm, draw] {}; 
      \draw[dotted,thick] (90+5*360/13:1.5cm) to (90-1*360/13:1.5cm);    
      \draw (90+5*360/13:1.75cm) node {$i$};
      \draw (90-1*360/13:1.8cm) node {$k$};
      \draw (90-0*360/13:1.75cm) node {$t$};
      \draw (90-5*360/13:1.75cm) node {$v$};
      \draw (90+2*360/13:1.75cm) node {$s$};
      \draw (90-6*360/13:1.75cm) node {$u$};
      \begin{scope}[thick,decoration={
    markings,mark=at position 0.5 with {\arrow{>}}}] 
        \draw[postaction={decorate}] (90+5*360/13:1.5cm)--(90-6*360/13:1.5cm);
      \end{scope}
      \begin{scope}[thick,decoration={
    markings,mark=at position 0.4 with {\arrow{>}}}]       
        \draw[postaction={decorate},red] (90-6*360/13:1.5cm)--(90+2*360/13:1.5cm);
        \draw[postaction={decorate}] (90+2*360/13:1.5cm)--(90-5*360/13:1.5cm);
      \end{scope}
      \begin{scope}[thick,decoration={
    markings,mark=at position 0.5 with {\arrow{>}}}]       
        \draw[postaction={decorate},red] (90-5*360/13:1.5cm)--(90+0*360/13:1.5cm);
        \draw[postaction={decorate}] (90+0*360/13:1.5cm)--(90-1*360/13:1.5cm);
      \end{scope}
    \end{scope}

    \begin{scope}[shift={(-1.9,0)}]
      \node[name=s, shape=regular polygon, regular polygon sides=13, minimum size=3cm, draw] {}; 
      \draw[dotted,thick] (90+5*360/13:1.5cm) to (90-1*360/13:1.5cm);    
      \draw (90+5*360/13:1.75cm) node {$i$};
      \draw (90-1*360/13:1.8cm) node {$k$};
      \draw (90-0*360/13:1.75cm) node {$t$};
      \draw (90-5*360/13:1.75cm) node {$v$};
      \draw (90+2*360/13:1.75cm) node {$s$};
      \draw (90-6*360/13:1.75cm) node {$u$};
      \begin{scope}[thick,decoration={
    markings,mark=at position 0.5 with {\arrow{>}}}] 
        \draw[postaction={decorate}] (90+5*360/13:1.5cm)--(90-6*360/13:1.5cm);
      \end{scope}
      \begin{scope}[thick,decoration={
    markings,mark=at position 0.4 with {\arrow{>}}}]       
        \draw[postaction={decorate},red] (90-6*360/13:1.5cm)--(90+2*360/13:1.5cm);
      \end{scope}
      \begin{scope}[thick,decoration={
    markings,mark=at position 0.5 with {\arrow{>}}}]       
        \draw[postaction={decorate}] (90+2*360/13:1.5cm)--(90+0*360/13:1.5cm);
        \draw[postaction={decorate},red] (90+0*360/13:1.5cm)--(90-5*360/13:1.5cm);
        \draw[postaction={decorate}] (90-5*360/13:1.5cm)--(90-1*360/13:1.5cm);
      \end{scope}
    \end{scope}

    \begin{scope}[shift={(1.9,0)}]
      \node[name=s, shape=regular polygon, regular polygon sides=13, minimum size=3cm, draw] {}; 
      \draw[dotted,thick] (90+5*360/13:1.5cm) to (90-1*360/13:1.5cm);    
      \draw (90+5*360/13:1.75cm) node {$i$};
      \draw (90-1*360/13:1.8cm) node {$k$};
      \draw (90-0*360/13:1.75cm) node {$t$};
      \draw (90-5*360/13:1.75cm) node {$v$};
      \draw (90+2*360/13:1.75cm) node {$s$};
      \draw (90-6*360/13:1.75cm) node {$u$};
      \begin{scope}[thick,decoration={
    markings,mark=at position 0.5 with {\arrow{>}}}] 
        \draw[postaction={decorate}] (90+5*360/13:1.5cm)--(90+2*360/13:1.5cm);
      \end{scope}
      \begin{scope}[thick,decoration={
    markings,mark=at position 0.4 with {\arrow{>}}}]       
        \draw[postaction={decorate},red] (90+2*360/13:1.5cm)--(90-6*360/13:1.5cm);
        \draw[postaction={decorate}] (90-6*360/13:1.5cm)--(90+0*360/13:1.5cm);
      \end{scope}
      \begin{scope}[thick,decoration={
    markings,mark=at position 0.5 with {\arrow{>}}}]       
        \draw[postaction={decorate},red] (90+0*360/13:1.5cm)--(90-5*360/13:1.5cm);
        \draw[postaction={decorate}] (90-5*360/13:1.5cm)--(90-1*360/13:1.5cm);
      \end{scope}
    \end{scope}

    \begin{scope}[shift={(5.7,0)}]
      \node[name=s, shape=regular polygon, regular polygon sides=13, minimum size=3cm, draw] {}; 
      \draw[dotted,thick] (90+5*360/13:1.5cm) to (90-1*360/13:1.5cm);    
      \draw (90+5*360/13:1.75cm) node {$i$};
      \draw (90-1*360/13:1.8cm) node {$k$};
      \draw (90-0*360/13:1.75cm) node {$t$};
      \draw (90-5*360/13:1.75cm) node {$v$};
      \draw (90+2*360/13:1.75cm) node {$s$};
      \draw (90-6*360/13:1.75cm) node {$u$};
      \begin{scope}[thick,decoration={
    markings,mark=at position 0.5 with {\arrow{>}}}] 
        \draw[postaction={decorate}] (90+5*360/13:1.5cm)--(90+2*360/13:1.5cm);
      \end{scope}
      \begin{scope}[thick,decoration={
    markings,mark=at position 0.4 with {\arrow{>}}}]       
        \draw[postaction={decorate},red] (90+2*360/13:1.5cm)--(90-6*360/13:1.5cm);
      \end{scope}
      \begin{scope}[thick,decoration={
    markings,mark=at position 0.5 with {\arrow{>}}}]       
        \draw[postaction={decorate}] (90-6*360/13:1.5cm)--(90-5*360/13:1.5cm);
        \draw[postaction={decorate},red] (90-5*360/13:1.5cm)--(90+0*360/13:1.5cm);
        \draw[postaction={decorate}] (90+0*360/13:1.5cm)--(90-1*360/13:1.5cm);
      \end{scope}
    \end{scope}

  \end{tikzpicture} 
\]
\endgroup
\caption{Top: A (red) polygon dissection $D = \{\, ( s,u ), ( u,s ), ( t,v ), ( v,t ) \,\}$ and a diagonal $( i,k )$.  Bottom: The $T$-paths from $i$ to $k$ with respect to $D$.  See Definition \ref{def:T-path}.}
\label{fig:exa:T-paths:a}
\end{figure}
If $c : \diag P \xrightarrow{} R$ is a map such that $c_{ su }$, $c_{ us }$, $c_{ tv }$, $c_{ vt }$ are in $R^*$, then $c$ satisfies the $T$-path formula \eqref{equ:def:T-path_formula:a} from $i$ to $k$ with respect to $D$ if
\[
  c_{ ik } =
  c_{ iu }c_{ su }^{ -1 }c_{ sv }c_{ tv }^{ -1 }c_{ tk } +
  c_{ iu }c_{ su }^{ -1 }c_{ st }c_{ vt }^{ -1 }c_{ vk } +
  c_{ is }c_{ us }^{ -1 }c_{ ut }c_{ vt }^{ -1 }c_{ vk } +
  c_{ is }c_{ us }^{ -1 }c_{ uv }c_{ tv }^{ -1 }c_{ tk }.
\]
By Theorem \ref{thm:B}, this holds if $c$ is a weak frieze with respect to $D$.  In particular, it holds if $c$ is a frieze.
\end{Example}

\section{Preliminary results}
\label{sec:preliminary}

\begin{Lemma}
[Alternative definition of friezes]
\label{lem:frieze}
In Definition \ref{def:frieze}, conditions (i) and (ii) are equivalent to the following conditions, respectively:
\begingroup
\renewcommand{\labelenumi}{(\roman{enumi}')}
\begin{enumerate}
\setcounter{enumi}{0}

  \item    If $i < j < k$ are pairwise different vertices of $P$, then the triangle relation
\[
  c_{ ij }c_{ kj }^{ -1 }c_{ ki } = c_{ ik }c_{ jk }^{ -1 }c_{ ji }  
\]
is satisfied.

  \item    If $i < j < k < \ell$ are pairwise different vertices of $P$, then the exchange relation
\[
  c_{ ik } = c_{ ij }c_{ \ell j }^{ -1 }c_{ \ell k }
           + c_{ i \ell }c_{ j \ell }^{ -1 }c_{ jk }
\]
is satisfied.
  
\end{enumerate}
\endgroup
\end{Lemma}

\begin{proof}
The conditions in Definition \ref{def:frieze} are stronger than their counterparts in the lemma.

Conversely, if $i$, $j$, $k$ are pairwise different vertices of $P$, then either $i < j < k$, in which case (i') is simply Definition \ref{def:frieze}(i), or $i < k < j$, in which case (i') gives
\[
  c_{ ik }c_{ jk }^{ -1 }c_{ ji } = c_{ ij }c_{ kj }^{ -1 }c_{ ki },
\]
yielding Definition \ref{def:frieze}(i) by swapping the left and right hand sides.

Likewise, if $( i,k )$ and $( j,\ell )$ cross, then $i$, $j$, $k$, $\ell$ are pairwise different vertices of $P$ and either $i < j < k < \ell$, in which case (ii') is simply Definition \ref{def:frieze}(ii), or $i < \ell < k < j$, in which case (ii') gives
\[
  c_{ ik } = c_{ i\ell }c_{ j\ell }^{ -1 }c_{ jk }
           + c_{ ij }c_{ \ell j }^{ -1 }c_{ \ell k },
\]
yielding Definition \ref{def:frieze}(ii) by swapping the terms on the right hand side.
\end{proof}

\begin{Lemma}
[Alternative definition of weak friezes]
\label{lem:weak_frieze}
In Definition \ref{def:weak_frieze}, condition (ii) is e\-qui\-va\-lent to the following:
\begingroup
\renewcommand{\labelenumi}{(\roman{enumi}')}
\begin{enumerate}
\setcounter{enumi}{1}

  \item  If $i < j < k < \ell$ are pairwise different vertices of $P$ and $( j,\ell )$, $( \ell,j )$ are in $D$, then the exchange relation
\[
  c_{ ik } = c_{ ij }c_{ \ell j }^{ -1 }c_{ \ell k }
           + c_{ i \ell }c_{ j \ell }^{ -1 }c_{ jk }
\]
is satisfied.

\end{enumerate}
\endgroup
\end{Lemma}

\begin{proof}
Similar to the proof of Lemma \ref{lem:frieze}.
\end{proof}

\begin{Lemma}
[Inverted triangle relation]
\label{lem:inverted_triangle_relations}
Let $P$ be a polygon, $c : \diag P \xrightarrow{} R$ a map, and $i$, $j$, $k$ pairwise different vertices of $P$.  Suppose that $c_{ jk }$, $c_{ kj }$ are in $R^*$ and that the triangle relation \eqref{equ:def:frieze:a} is satisfied.  If $c_{ ij }$, $c_{ ji }$ are in $R^*$, then the {\em inverted triangle relation}
\[
  c_{ ij }^{ -1 }c_{ ik }c_{ jk }^{ -1 } = c_{ kj }^{ -1 }c_{ ki }c_{ ji }^{ -1 }
\]
is satisfied.
\end{Lemma}

\begin{proof}
Swap the left and right hand sides of \eqref{equ:def:frieze:a} and multiply by $c_{ ij }^{ -1 }$ from the left and $c_{ ji }^{ -1 }$ from the right.
\end{proof}

\begin{Proposition}
\label{pro:A}
Suppose that $R$ is a skew field.  Let the following be given.
\medskip
\begin{itemize}
\setlength\itemsep{4pt}

  \item  $P$ is a polygon divided into subpolygons $P^1, \ldots, P^{ \kappa+1 }$ by the dissection
\begin{equation}
\label{equ:pro:A:b}
  \{\, ( t_1,v_1 ),( v_1,t_1 ), \ldots, ( t_{ \kappa },v_{ \kappa } ), ( v_{ \kappa },t_{ \kappa } ) \,\}
\end{equation}
for some $\kappa \geqslant 1$.  See Figure \ref{fig:thm:C:a}.
  
  \item  $c : \diag P \xrightarrow{} R$ is a weak frieze with respect to the dissection \eqref{equ:pro:A:b}.
  
\end{itemize}
\medskip
Set $c^{ \alpha } = c \Big|_{ \diag P^{ \alpha } }$ for $\alpha$ in $\{\, 1, \ldots, \kappa+1 \,\}$.  Then the Dieudonn\'{e} determinants satisfy
\[
  \det M_c = 
  \overline{ ( -1 )^{ \kappa } }
  \cdot 
  \mathlarger{ \prod }_{ \alpha=1 }^{ \kappa }\:
    \overline{ c_{ t_{ \alpha }v_{ \alpha } }^{ -1 }c_{ v_{ \alpha }t_{ \alpha } }^{ -1 } }  
  \cdot
  \mathlarger{ \prod }_{ \alpha=1 }^{ \kappa+1 }
    \det M_{ c^{ \alpha } },
\]
where the matrices $M_c$ and $M_{ c^{ \alpha } }$ are defined by Figure \ref{fig:frieze_matrix}.  Recall that $\overline{ c }$ means the coset in the abelianisation $R^*/[ R^*,R^* ]$ of the element $c$ in $R^*$.  
\end{Proposition}

\begin{proof}
We will assume $\kappa = 1$ since the general case follows from this by an easy induction.  Then there are two subpolygons, $P^1$ and $P^2$, and the dissection \eqref{equ:pro:A:b} has the form $\{\, ( t,v ), ( v,t ) \,\}$.  Let $m$ and $n$ be such that $P^1$ is an $m$-gon and $P^2$ is an $n$-gon; then $P$ is an $( m+n-2 )$-gon whose vertices will be numbered as shown in Figure \ref{fig:pro:A:a} such that $t = m-1$ and $v = m$.
\begin{figure}
\begingroup
\[
  \begin{tikzpicture}[scale=1.3333333]
      \node[name=s, shape=regular polygon, regular polygon sides=20, rotate=9, minimum size=8cm, draw] {}; 
      \draw [thick] (5*360/20:3cm) to (15*360/20:3cm);
      \draw (6*360/20:3.3cm) node {$1$};
      \draw (7*360/20:3.3cm) node {$2$};
      \draw (7*360/20+11:3.3cm) node {$\cdot$};
      \draw (7*360/20+14:3.3cm) node {$\cdot$};
      \draw (7*360/20+17:3.3cm) node {$\cdot$};
      \draw (13*360/20-22:3.3cm) node {$\cdot$};
      \draw (13*360/20-19:3.3cm) node {$\cdot$};
      \draw (13*360/20-16:3.3cm) node {$\cdot$};
      \draw (13*360/20-4:3.5cm) node {$m-3$};
      \draw (14*360/20-5:3.4cm) node {$m-2$};
      \draw (15*360/20:3.3cm) node {$t = m-1$};
      \draw (5*360/20:3.3cm) node {$v = m$};
      \draw (4*360/20-4:3.4cm) node {$m+1$};
      \draw (3*360/20-4:3.5cm) node {$m+2$};
      \draw (3*360/20-16:3.3cm) node {$\cdot$};
      \draw (3*360/20-19:3.3cm) node {$\cdot$};
      \draw (3*360/20-22:3.3cm) node {$\cdot$};
      \draw (17*360/20+22:3.3cm) node {$\cdot$};
      \draw (17*360/20+19:3.3cm) node {$\cdot$};
      \draw (17*360/20+16:3.3cm) node {$\cdot$};
      \draw (17*360/20+9:3.65cm) node {$m+n-3$};
      \draw (16*360/20+11:3.45cm) node {$m+n-2$};
      \draw (-1.2,0) node {$P^1$};
      \draw (1.2,0) node {$P^2$};
  \end{tikzpicture} 
\]
\endgroup
\caption{The proof of Proposition \ref{pro:A} considers a polygon $P$ split into subpolygons $P^1$ and $P^2$ by the dissection $\{\, ( t,v ), ( v,t ) \,\}$.  The vertices will be numbered as shown such that $t = m-1$ and $v = m$.}
\label{fig:pro:A:a}
\end{figure}
This gives the matrix $M_c$ shown in Figure \ref{fig:pro:A:b}.
\begin{figure}
\begingroup
\renewcommand{\arraystretch}{1.5}
\[
  M_c
  =
  \left(
    \begin{array}{ccc|cc|ccc}
      & & & * & * & & & \\
      & M^1 & & \vdots & \vdots & & * & \\
      & & & * & * & & & \\
\hline
      * & \cdots & * & 0 & c_{ m-1,m } & * & \cdots & * \\
      * & \cdots & * & c_{ m,m-1 } & 0 & * & \cdots & * \\
\hline
      & & & * & * & & & \\
      & * & & \vdots & \vdots & & M^2 & \\
      & & & * & * & & & \\
    \end{array}
  \right)
\]
\endgroup
\caption{The proof of Proposition \ref{pro:A} considers the matrix $M_c$ corresponding to the numbering in Figure \ref{fig:pro:A:a}.}
\label{fig:pro:A:b}
\end{figure}

Denote the central block by
$
C
=
\left(
  \begin{array}{cc}
    0 & c_{ m-1,m } \\
    c_{ m,m-1 } & 0
  \end{array}
\right).
$
Then the block with $M^1$ and $C$ on the diagonal is $M_{ c^1 }$, and the block with $C$ and $M^2$ on the diagonal is $M_{ c^2 }$.

Define a new matrix $\widetilde{M}_c$ by row operations on $M_c$, subtracting certain left multiples of rows number $m-1$ and $m$ from the first $m-2$ rows.  Specifically, define coefficients
\[
  x_i = c_{ im }c_{ m-1,m }^{ -1 }
  \;\;,\;\;
  y_i = c_{ i,m-1 }c_{ m,m-1 }^{ -1 }
\]
for $1 \leqslant i \leqslant m-2$ and let $\widetilde{M}_c$ have the entries
\[
  \widetilde{ c }_{ ik }
  =
  \left\{
    \begin{array}{ll}
      c_{ ik } - x_ic_{ m-1,k } - y_ic_{ m,k } & \mbox{ for } 1 \leqslant i \leqslant m-2, \\[2mm]
      c_{ ik } & \mbox{ for } m-1 \leqslant i \leqslant m+n-2. \\
    \end{array}
  \right.
\]
For $1 \leqslant i \leqslant m-2$ we have
\[
  \widetilde{ c }_{ ik }
  =
  c_{ ik } - x_ic_{ m-1,k } - y_ic_{ m,k }
  =
  c_{ ik } - c_{ im }c_{ m-1,m }^{ -1 }c_{ m-1,k } - c_{ i,m-1 }c_{ m,m-1 }^{ -1 }c_{ m,k }
  =
  (*).
\]
If $m-1 \leqslant k \leqslant m$ then $(*) = 0$ by direct computation in view of the convention $c_{ xx } = 0$.  If $m+1 \leqslant k \leqslant m+n-2$ then $(*) = 0$ by an exchange relation, which holds because $c$ is a weak frieze with respect to $\{\, ( t,v ),( v,t ) \,\} = \{\, ( m-1,m ),( m,m-1 ) \,\}$ and because $1 \leqslant i \leqslant m-2$ and $m+1 \leqslant k \leqslant m+n-2$ imply that $( i,k )$ and $( m-1,m )$ cross.  Hence $\widetilde{ M }_c$ is as shown in Figure \ref{fig:pro:A:c}.
\begin{figure}
\begingroup
\renewcommand{\arraystretch}{1.5}
\[
  \widetilde{ M }_c
  =
  \left(
    \begin{array}{ccc|cc|ccc}
      & & & 0 & 0 & & & \\
      & \widetilde{ M }^1 & & \vdots & \vdots & & 0 & \\
      & & & 0 & 0 & & & \\
\hline      
      * & \cdots & * & 0 & c_{ m-1,m } & * & \cdots & * \\
      * & \cdots & * & c_{ m,m-1 } & 0 & * & \cdots & * \\
\hline
      & & & * & * & & & \\
      & * & & \vdots & \vdots & & M^2 & \\
      & & & * & * & & & \\
    \end{array}
  \right)
\]
\endgroup
\caption{The proof of Proposition \ref{pro:A} considers the matrix $\widetilde{ M }_c$ obtained by row operations on the matrix $M_c$ in Figure \ref{fig:pro:A:b}.}
\label{fig:pro:A:c}
\end{figure}

Let $\widetilde{ M }_{ c^1 }$ denote the block with $\widetilde{ M }^1$ and $C$ on the diagonal.  We can compute as follows.
\begin{align}
\nonumber
  \det( C )^{ -1 }\det( M_{ c^1 } )\det( M_{ c^2 } )
  & \stackrel{ \rm (a) }{ = } \det( C )^{ -1 }\det( \widetilde{ M }_{ c^1 } )\det( M_{ c^2 } ) \\[2mm]
\nonumber
  & \stackrel{ \rm (b) }{ = } \det( C )^{ -1 }\det( \widetilde{ M }^1  )\det( C )\det( M_{ c^2 } ) \\[2mm]
\nonumber
  & = \det( \widetilde{ M }^1  )\det( M_{ c^2 } ) \\[2mm]
\nonumber
  & \stackrel{ \rm (c) }{ = } \det \widetilde{ M }_c \\[2mm]
\label{equ:pro:A:a}
  & \stackrel{ \rm (d) }{ = } \det M_c
\end{align}
Here (a) and (d) hold by \cite[p.\ 153, part a)]{Artin-book} since $\widetilde{ M }_{ c^1 }$ and $\widetilde{ M }_c$ differ from $M_{ c^1 }$ and $M_c$, respectively, by row operations, and (b) and (c) hold by \cite[thm.\ 4.4]{Artin-book} since $\widetilde{ M }_{ c^1 }$ and $\widetilde{ M }_c$ are lower diagonal block matrices.

Finally, $\det C = \overline{ -1 } \cdot \overline{ c_{ m-1,m }c_{ m,m-1 } } = \overline{ -1 } \cdot \overline{ c_{ tv }c_{ vt } }$ by \cite[p.\ 157, item 1)]{Artin-book}.  Combining with \eqref{equ:pro:A:a} gives
\[
  \overline{ -1 } \cdot \overline{ c_{ tv }^{ -1 }c_{ vt }^{ -1 } }
    \cdot \det( M_{ c^1 } )\det( M_{ c^2 } )
  = \det M_c
\]
as claimed.
\end{proof}

\begin{Setup}
\label{set:P1P2}
Lemmas \ref{lem:B} through \ref{lem:8} concern the following situation shown in Figure \ref{fig:lem:B:a}.  
\begin{itemize}

  \item  $P$ is a polygon divided into subpolygons $P^1$ and $P^2$ by the dissection 
\[
  \{\, ( t,v ), ( v,t ) \,\}.
\]
  
  \item  $P^{ \alpha }$ has a dissection $D^{ \alpha }$ for each $\alpha$ in $\{\, 1,2 \,\}$, and $P$ has the dissection
\[
  D = \{\, ( t,v ), ( v,t ) \,\} \cupdot D^1 \cupdot D^2.
\]
\begin{figure}
\begingroup
\[
  \begin{tikzpicture}[scale=0.75]
      \node[name=s, shape=regular polygon, regular polygon sides=20, rotate=9, minimum size=6cm, draw] {}; 
      \draw [thick] (7*360/20:4cm) to (16*360/20:4cm);
      \draw (7*360/20:4.3cm) node { $t$ };
      \draw (16*360/20:4.3cm) node { $v$ };
      \draw (11.5*360/20:1.2cm) node { $P^1$ };
      \draw (1.5*360/20:-0.05cm) node { $P^2$ };      
      \draw [thick,red] (10*360/20:4cm) to (12*360/20:4cm);
      \draw [thick,red] (9*360/20:4cm) to (14*360/20:4cm);
      \draw [thick,blue] (2*360/20:4cm) to (-1*360/20:4cm);
      \draw [thick,blue] (4*360/20:4cm) to (-2*360/20:4cm);
      \draw [red] (-2.3cm,-1.0cm) node { $\cdot$ };
      \draw [red] (-2.1cm,-0.9cm) node { $\cdot$ };
      \draw [red] (-1.9cm,-0.8cm) node { $\cdot$ };      
      \draw [blue] (1.4cm,0.8cm) node { $\cdot$ };
      \draw [blue] (1.6cm,0.9cm) node { $\cdot$ };
      \draw [blue] (1.8cm,1.0cm) node { $\cdot$ };      
  \end{tikzpicture} 
\]
\endgroup
\caption{Setup \ref{set:P1P2} considers a polygon $P$ divided into subpolygons $P^1$ and $P^2$ by the dissection $\{\, ( t,v ), ( v,t ) \,\}$.  The subpolygons in turn have dissections $D^1$ (red) and $D^2$ (blue), and $P$ has the dissection $D = \{\, ( t,v ), ( v,t ) \,\} \cupdot D^1 \cupdot D^2$.}
\label{fig:lem:B:a}
\end{figure}
\end{itemize}
We will use the following notation for each $\alpha$ in $\{\, 1,2 \,\}$.
\begin{itemize}
\setlength\itemsep{4pt}

  \item  $V^{ \alpha }$ is the set of vertices of $P^{ \alpha }$.
  
  \item  $U^{ \alpha } = V^{ \alpha } \setminus \{\, t,v \,\}$ is the set of vertices of $P^{ \alpha }$ not belonging to the other subpolygon.
  
\end{itemize}
\end{Setup}

\begin{Lemma}
\label{lem:B}
Consider Setup \ref{set:P1P2} and let the following additional data be given:
\begin{itemize}
\setlength\itemsep{4pt}
  
  \item  $c^{ \alpha } : \diag P^{ \alpha } \xrightarrow{} R$ is a weak frieze with respect to $D^{ \alpha }$ for each $\alpha$ in $\{\, 1,2 \,\}$.

\end{itemize}
Assume that:
\begin{itemize}

  \item  We have $c^1_{ tv } = c^2_{ tv }$ and $c^1_{ vt } = c^2_{ vt }$ and these elements are in $R^*$.  

\end{itemize}
Then there exists a unique map $c : \diag P \xrightarrow{} R$ satisfying the following conditions.
\begin{enumerate}
\setlength\itemsep{4pt}

  \item  $c \Big|_{ \diag P^{ \alpha } } = c^{ \alpha }$ for each $\alpha$ in $\{\, 1,2 \,\}$. 

  \item  $c$ is a weak frieze with respect to $D$.

\end{enumerate}
Each $c_{ ik }$ is given by the relevant entry of the following table where $U^1$, $U^2$ are as defined in Setup \ref{set:P1P2}.
\medskip
\[
\mbox{
\begin{tabular}{c|cccccc}
  \diaghead{\theadfont xxxxxxxxx}{$i$}{$k$} & $U^1$ & $\{ t,v \}$ & $U^2$ \\[2mm] \cline{1-4}
  \vphantom{$U^{1^{1^1}}$}$U^1$ & $c^1_{ ik }$ & $c^1_{ ik }$ & $c^1_{ it }( c^1_{ vt } )^{ -1 }c^2_{ vk } + c^1_{ iv }( c^1_{ tv } )^{ -1 }c^2_{ tk }$ \\[2mm]
  $\{ t,v \}$ & $c^1_{ ik }$ & $c^1_{ ik } = c^2_{ ik }$ & $c^2_{ ik }$ \\[2mm]
  $U^2$ & $c^2_{ it }( c^2_{ vt } )^{ -1 }c^1_{ vk } + c^2_{ iv }( c^2_{ tv } )^{ -1 }c^1_{ tk }$ & $c^2_{ ik }$ & $c^2_{ ik }$ \\
\end{tabular}
}
\]
\end{Lemma}

\begin{proof}
Let $c : \diag P \xrightarrow{} R$ be the map defined by the table.

It is clearly the unique weak frieze on $P$ with respect to the dissection $\{\, ( t,v ), ( v,t ) \,\}$ which satisfies condition (i) of the Lemma.  We will complete the proof by showing that $c$ satisfies Definition \ref{def:weak_frieze} with respect to the dissection $D$.

Definition \ref{def:weak_frieze}(i) follows from the definition of $c$ and the conditions on the $c^{ \alpha }$.

To show that Definition \ref{def:weak_frieze}(ii) is satisfied, it is enough to show that Lemma \ref{lem:weak_frieze}(ii') is satisfied, so let $i < j < k < \ell$ be pairwise different vertices of $P$ with $( j,\ell )$, $( \ell,j )$ in $D$.  We must establish the exchange relation
\begin{equation}
\label{equ:lem:B:a}
  c_{ ik } = c_{ ij }c_{ \ell j }^{ -1 }c_{ \ell k }
           + c_{ i\ell }c_{ j\ell }^{ -1 }c_{ jk }.
\end{equation}
If $\{\, ( j,\ell ),( \ell,j ) \,\}$ is $\{\, ( t,v ),( v,t ) \,\}$, then \eqref{equ:lem:B:a} is true by the definition of $c$.  If $( j,\ell )$, $( \ell,j )$ are in $D^1$ or in $D^2$, we can assume by symmetry that they are in $D^1$ and split into the following cases.  Note that $j$, $\ell$ are in $V^1$ so $i$, $k$ cannot both be in $V^2$ because $i < j < k < \ell$.

Case 1: $i$, $k$ are in $V^1$.  Then $i < j < k < \ell$ are pairwise different vertices of $P^1$.  By definition of $c$, Equation \eqref{equ:lem:B:a} amounts to the same equation with $c^1$ instead of $c$, and this equation is true because $c^1$ is a weak frieze with respect to $D^1$.  

Case 2: 
$i$ is in $U^1$ and $k$ is in $U^2$.  

The definition of $c$ implies
\begin{equation}
\label{equ:lem:B:b}
  c_{ ik } = c_{ it }c_{ vt }^{ -1 }c_{ vk }
           + c_{ iv }c_{ tv }^{ -1 }c_{ tk }.
\end{equation}
If the diagonals $( j,k )$ and $( t,v )$ cross, then the definition of $c$ also implies
\begin{equation}
\label{equ:lem:B:c}
  c_{ jk } = c_{ jt }c_{ vt }^{ -1 }c_{ vk }
           + c_{ jv }c_{ tv }^{ -1 }c_{ tk }.
\end{equation}
The last equation is also true if the diagonals $( j,k )$ and $( t,v )$ do not cross.  The reason is that in this case, $j$ must be $t$ or $v$ since $j$ is in $V^1$ while $k$ is in $U^2$, and we have the convention $c_{ xx } = 0$.  Similarly we get
\begin{equation}
\label{equ:lem:B:d}
  c_{ \ell k } = c_{ \ell t }c_{ vt }^{ -1 }c_{ vk }
           + c_{ \ell v }c_{ tv }^{ -1 }c_{ tk }.
\end{equation}

The restriction of $c$ to $P^1$ is a weak frieze with respect to $\{\, ( j,\ell ),( \ell,j ) \,\}$ by assumption, so if the diagonals $( i,t )$ and $( j,\ell )$ cross, then
\begin{equation}
\label{equ:lem:B:e}
  c_{ it } = c_{ ij }c_{ \ell j }^{ -1 }c_{ \ell t }
           + c_{ i\ell }c_{ j\ell }^{ -1 }c_{ jt }.
\end{equation}
The last equation is also true if the diagonals $( i,t )$ and $( j,\ell )$ do not cross.  The reason is that in this case, $j$ or $\ell$ must be $t$ since $j$ and $\ell$ are in $V^1$; see Figure \ref{fig:lem:B:b}.
\begin{figure}
\begingroup
\[
  \begin{tikzpicture}[scale=1]
      \node[name=s, shape=regular polygon, regular polygon sides=20, rotate=9, minimum size=5cm, draw] {}; 
      \draw [thick] (7*360/20:2.5cm) to (16*360/20:2.5cm);
      \draw [dotted,thick] (7*360/20:2.5cm) to (12*360/20:2.5cm);
      \draw [dotted,thick] (9*360/20:2.5cm) to (14*360/20:2.5cm);
      \draw (7*360/20:2.75cm) node { $t$ };
      \draw (16*360/20:2.75cm) node { $v$ };
      \draw (12*360/20:2.7cm) node { $i$ };
      \draw (14*360/20:2.80cm) node { $j$ };
      \draw (1*360/20:2.75cm) node { $k$ };
      \draw (9*360/20:2.7cm) node { $\ell$ };
      \draw (11.5*360/20:0.8cm) node { $P^1$ };
      \draw (1.5*360/20:0.05cm) node { $P^2$ };      
  \end{tikzpicture} 
\]
\endgroup
\caption{In the proof of Lemma \ref{lem:B}, Case 2, if the diagonals $( i,t )$ and $( j,\ell )$ do not cross then we must have $\ell = t$ (if $t$ and $v$ are placed as here) or $j = t$ (if $t$ and $v$ are swapped).}
\label{fig:lem:B:b}
\end{figure}
Similarly we get
\begin{equation}
\label{equ:lem:B:f}
  c_{ iv } = c_{ ij }c_{ \ell j }^{ -1 }c_{ \ell v }
           + c_{ i\ell }c_{ j\ell }^{ -1 }c_{ jv }.
\end{equation}

We can now prove \eqref{equ:lem:B:a} by the following computation:
\begin{align*}
  c_{ ij }c_{ \ell j }^{ -1 }c_{ \ell k }
  + c_{ i\ell }c_{ j\ell }^{ -1 }c_{ jk }
  & \stackrel{ \rm (a) }{ = }
  c_{ ij }c_{ \ell j }^{ -1 }
  ( c_{ \ell t }c_{ vt }^{ -1 }c_{ vk }
    + c_{ \ell v }c_{ tv }^{ -1 }c_{ tk } )
  + c_{ i\ell }c_{ j\ell }^{ -1 }
  ( c_{ jt }c_{ vt }^{ -1 }c_{ vk }
    + c_{ jv }c_{ tv }^{ -1 }c_{ tk } ) \\[2mm]
  & =
  ( c_{ ij }c_{ \ell j }^{ -1 }c_{ \ell t }
    + c_{ i\ell }c_{ j\ell }^{ -1 }c_{ jt } )c_{ vt }^{ -1 }c_{ vk }
  +
  ( c_{ ij }c_{ \ell j }^{ -1 }c_{ \ell v }
    + c_{ i\ell }c_{ j\ell }^{ -1 }c_{ jv } )c_{ tv }^{ -1 }c_{ tk } \\[2mm]
  & \stackrel{ \rm (b) }{ = }
  c_{ it }c_{ vt }^{ -1 }c_{ vk }
  +
  c_{ iv }c_{ tv }^{ -1 }c_{ tk } \\[2mm]
  & \stackrel{ \rm (c) }{ = }
  c_{ ik }.
\end{align*}
Here (a) is by \eqref{equ:lem:B:c} and \eqref{equ:lem:B:d} while (b) is by \eqref{equ:lem:B:e} and \eqref{equ:lem:B:f}.  Finally, (c) is by \eqref{equ:lem:B:b}.

Case 3: 
$i$ is in $U^2$ and $k$ is in $U^1$.  This is proved similarly to Case 2.
\end{proof}

%

\begin{Lemma}
\label{lem:computation}
Consider Setup \ref{set:P1P2} with $D^1 = \emptyset$ and let $c : \diag P \xrightarrow{} R$ be a map such  that $c_{ rs }$ is in $R^*$ if $( r,s )$ is in $D$.

For $i$ in $U^1$ and $k$ in $U^2$ we have
\begin{equation}
\label{equ:lem:computation:a}
  \sum_{ p \in \cP( D,i,k ) } c_p
  =
  c_{ it }c_{ vt }^{ -1 } \sum_{ q \in \cP( D^2,v,k ) } c_q
  +
  c_{ iv }c_{ tv }^{ -1 } \sum_{ q \in \cP( D^2,t,k ) } c_q
\end{equation}
For $i$ in $U^2$ and $k$ in $U^1$ we have
\begin{equation}
\label{equ:lem:computation:u}
  \sum_{ p \in \cP( D,i,k ) } c_p
  =
  \Bigg( \sum_{ q \in \cP( D^2,i,t ) } c_q \Bigg)
  c_{ vt }^{ -1 }c_{ vk }
  +
  \Bigg( \sum_{ q \in \cP( D^2,i,v ) } c_q \Bigg)
  c_{ tv }^{ -1 }c_{ tk }.
\end{equation}

\end{Lemma}

\begin{proof}
First, suppose that $i$ is in $U^1$ and $k$ is in $U^2$.  We will prove Equation \eqref{equ:lem:computation:a}.  By \cite[lemmas 3.4 and 3.8]{Canakci-Jorgensen} there are bijections
\begin{align*}
  \cP( D,i,k )^{ ( i,t,\neq v ) } 
    & \xrightarrow{R} \cP( D^2,v,k )^{ ( v,t ) }, \\[2mm]
  ( i,t,p_3,\ldots,p_{ \pi-1 },k )
  & \mapsto ( v,t,p_3,\ldots,p_{ \pi-1 },k )
\end{align*}
and
\begin{align*}
  \cP( D,i,k )^{ ( i,t,v ) } 
    & \xrightarrow{S} \cP( D^2,v,k )^{ ( v,\neq t ) }, \\[2mm]
  ( i,t,v,p_4,\ldots,p_{ \pi-1 },k )
  & \mapsto ( v,p_4,\ldots,p_{ \pi-1 },k ).
\end{align*}
This provides the equality
\[
  \sum_{ p \in \cP( D,i,k )^{ ( i,t,\neq v ) } } c_{ R( p ) }
    + \sum_{ p \in \cP( D,i,k )^{ ( i,t,v ) } } c_{ S( p ) }
  = \sum_{ q \in \cP( D^2,v,k )^{ ( v,t ) } } c_q
    + \sum_{ q \in \cP( D^2,v,k )^{ ( v,\neq t ) } } c_q,
\]
which clearly implies
\begin{equation}
\label{equ:lem:computation:b}	
  \sum_{ p \in \cP( D,i,k )^{ ( i,t,\neq v ) } } c_{ R( p ) }
    + \sum_{ p \in \cP( D,i,k )^{ ( i,t,v ) } } c_{ S( p ) }
  = \sum_{ q \in \cP( D^2,v,k ) } c_q.
\end{equation}
It is immediate from the definitions that for $p$ in $\cP( D,i,k )^{ ( i,t,\neq v ) }$ we have
\[
  c_p = c_{ it }c_{ vt }^{ -1 }c_{ R( p ) }
\]
and for $p$ in $\cP( D,i,k )^{ ( i,t,v ) }$ we have
\[
  c_p = c_{ it }c_{ vt }^{ -1 }c_{ S( p ) }.
\]
Multiplying Equation \eqref{equ:lem:computation:b} by $c_{ it }c_{ vt }^{ -1 }$ hence gives
\[
  \sum_{ p \in \cP( D,i,k )^{ ( i,t,\neq v ) } } c_p
    + \sum_{ p \in \cP( D,i,k )^{ ( i,t,v ) } } c_p
  = c_{ it }c_{ vt }^{ -1 } \sum_{ q \in \cP( D^2,v,k ) } c_q,
\]
which clearly implies
\begin{equation}
\label{equ:lem:computation:c}
  \sum_{ p \in \cP( D,i,k )^{ ( i,t ) } } c_p
  = c_{ it }c_{ vt }^{ -1 } \sum_{ q \in \cP( D^2,v,k ) } c_q.
\end{equation}
Symmetrically, we have
\begin{equation}
\label{equ:lem:computation:d}
  \sum_{ p \in \cP( D,i,k )^{ ( i,v ) } } c_p
  = c_{ iv }c_{ tv }^{ -1 }
    \sum_{ q \in \cP( D^2,t,k ) } c_q.
\end{equation}
By \cite[lem.\ 3.7]{Canakci-Jorgensen}, the disjoint union of $\cP( D,i,k )^{ ( i,t ) }$ and $\cP( D,i,k )^{ ( i,v ) }$ is $\cP( D,i,k )$, so adding Equations \eqref{equ:lem:computation:c} and \eqref{equ:lem:computation:d} gives Equation \eqref{equ:lem:computation:a}.

Secondly, suppose that $i$ is in $U^2$ and $k$ is in $U^1$.  We will prove Equation \eqref{equ:lem:computation:u}.  By \cite[lem.\ 3.1(i)]{Canakci-Jorgensen} there is a bijection $\cP( D,i,k ) \xrightarrow{} \cP( D,k,i )$ given by $p \mapsto \overline{ p }$ with $\overline{ p } = ( p_{ \pi }, \ldots, p_1 )$ for each $T$-path $p = ( p_1, \ldots, p_{ \pi } )$.  We define $\overline{ c } : \diag P \xrightarrow{} R^{ \opp }$ by $\overline{ c }_{ ik } = c_{ ki }$, where $R^{ \opp }$ is the opposite ring with multiplication defined by $c \times d = dc$.  A straightforward computation shows
\begin{equation}
\label{equ:lem:computation:v}
  \overline{ c }_p = c_{ \overline{ p } }.
\end{equation}
We now prove Equation \eqref{equ:lem:computation:u} as follows.
\begin{align*}
  \sum_{ p \in \cP( D,i,k ) } c_p
  & \stackrel{ \rm (a) }{ = } \sum_{ r \in \cP( D,k,i ) } c_{ \overline{ r } } \\[2mm]
  & \stackrel{ \rm (b) }{ = } \sum_{ r \in \cP( D,k,i ) } \overline{ c }_r \\[2mm]
  & \stackrel{ \rm (c) }{ = }
  \overline{ c }_{ kt } \times \overline{ c }_{ vt }^{ -1 } \times \sum_{ s \in \cP( D^2,v,i ) } \overline{ c }_s
  +
  \overline{ c }_{ kv } \times \overline{ c }_{ tv }^{ -1 } \times \sum_{ s \in \cP( D^2,t,i ) } \overline{ c }_s \\[2mm]
  & \stackrel{ \rm (d) }{ = }
  \Bigg( \sum_{ s \in \cP( D^2,v,i ) } c_{ \overline{ s }} \Bigg) c_{ tv }^{ -1 } c_{ tk }
  +
  \Bigg( \sum_{ s \in \cP( D^2,t,i ) } c_{ \overline{ s } } \Bigg) c_{ vt }^{ -1 } c_{ vk } \\[2mm]
  & \stackrel{ \rm (e) }{ = }
  \Bigg( \sum_{ q \in \cP( D^2,i,v ) } c_q \Bigg) c_{ tv }^{ -1 } c_{ tk }
  +
  \Bigg( \sum_{ q \in \cP( D^2,i,t ) } c_q \Bigg) c_{ vt }^{ -1 } c_{ vk }
\end{align*}
Here (a) and (e) are by the bijection $p \mapsto \overline{p}$, and (b) is by Equation \eqref{equ:lem:computation:v}.  Equality (c) is by Equation \eqref{equ:lem:computation:a} with $i$ and $k$ swapped applied to $\overline{ c }$, and (d) is by $c \times d = dc$ and Equation \eqref{equ:lem:computation:v}.
\end{proof}

\begin{Lemma}
\label{lem:8}
Consider Setup \ref{set:P1P2} with $D^1 = \emptyset$ and let $c : \diag P \xrightarrow{} R$ be a map such  that $c_{ rs }$ is in $R^*$ if $( r,s )$ is in $D$.

Assume that:
\begin{itemize}
\setlength\itemsep{4pt}

  \item  $c \Big|_{ \diag P^2 }$ satisfies the $T$-path formula \eqref{equ:def:T-path_formula:a} with respect to $D^2$.
  
  \item  $c$ is a weak frieze with respect to $\{\, ( t,v ), ( v,t ) \,\}$.

\end{itemize}
Then $c$ satisfies the $T$-path formula from $i$ to $k$ with respect to $D$ when $i$ is in $U^1$ and $k$ is in $U^2$.
\end{Lemma}

\begin{proof}
Let $i$ be in $U^1$ and $k$ be in $U^2$.  Lemma \ref{lem:computation} provides Equation \eqref{equ:lem:computation:a}.  Since $c \Big|_{ \diag P^2 }$ satisfies the $T$-path formula with respect to $D^2$, the equation says

\[
  \sum_{ p \in \cP( D,i,k ) } c_p
  =
  c_{ it }c_{ vt }^{ -1 }c_{ vk }
  +
  c_{ iv }c_{ tv }^{ -1 }c_{ tk },
\]
and since $c$ is a weak frieze with respect to $\{\, ( t,v ), ( v,t ) \,\}$, this equation says
\[
  \sum_{ p \in \cP( D,i,k ) } c_p
  =
  c_{ ik }
\]
as claimed.
\end{proof}

\section{Proofs of the Theorems \ref{thm:A}, \ref{thm:B}, \ref{thm:C}, \ref{thm:D}}
\label{sec:proofs}

\begin{proof}
[Proof of Theorem \ref{thm:A} (Non-commutative frieze determinants)]
We use induction on $n$, the number of vertices of the polygon $P$ on which the frieze $c$ is defined.

$n = 3$:  We have
\[
  M_c =
  \begin{pmatrix}
    0        & c_{ 12 } & c_{ 13 } \\[2mm]
    c_{ 21 } & 0        & c_{ 23 } \\[2mm]
    c_{ 31 } & c_{ 32 } & 0
  \end{pmatrix}.
\]
We flip rows one and two to get
\[
  M'_c =
  \begin{pmatrix}
    c_{ 21 } & 0        & c_{ 23 } \\[2mm]
    0        & c_{ 12 } & c_{ 13 } \\[2mm]
    c_{ 31 } & c_{ 32 } & 0
  \end{pmatrix},
\]
then perform two row operations by subtracting $c_{ 31 }c_{ 21 }^{ -1 }$ times row one and $c_{ 32 }c_{ 12 }^{ -1 }$ times row two from row three.  This gives
\[
  M''_c =
  \begin{pmatrix}
    c_{ 21 } & 0         & c_{ 23 } \\[2mm]
    0        & c_{ 12 }  & c_{ 13 } \\[2mm]
    0        & 0         & -c_{ 31 }c_{ 21 }^{ -1 }c_{ 23 }-c_{ 32 }c_{ 12 }^{ -1 }c_{ 13 }
  \end{pmatrix}.
\]
Using \cite[p.\ 153, parts a) and c) and thm.\ 4.4]{Artin-book} gives
\[
  \det M_c
  = \overline{ -1 } \cdot \det M''_c
  = \overline{ -1 } 
    \cdot \overline{ c_{ 21 } }
    \cdot \overline{ c_{ 12 } }
    \cdot \overline{ -c_{ 31 }c_{ 21 }^{ -1 }c_{ 23 }-c_{ 32 }c_{ 12 }^{ -1 }c_{ 13 } }
  = (*).
\]
The triangle relation \eqref{equ:def:frieze:a} gives the first of the following equalities, and the second holds because $( R^* / [ R^*,R^* ] ) \cupdot \{\, 0 \,\}$ is abelian.
\[
  (*)
  = \overline{ -1 } 
    \cdot \overline{ c_{ 21 } }
    \cdot \overline{ c_{ 12 } }
    \cdot \overline{ -2c_{ 31 }c_{ 21 }^{ -1 }c_{ 23 } }
  = \overline{ 2 } \cdot \overline{ c_{ 12 }c_{ 23 }c_{ 31 } }.
\]

$n \geqslant 4$: Assume the theorem holds for smaller values of $n$.  Since $c$ is a frieze, it is a weak frieze with respect to any dissection of $P$ as remarked in the introduction.  In particular, it is a weak frieze with respect to the dissection $\{\, ( 1,n-1 ), ( n-1,1 ) \,\}$, which divides $P$ into a subpolygon $P^1$ with vertex set $V^1$ numbered by $\{\, 1, \ldots, n-1 \,\}$ and a triangle $P^2$ with vertex set $V^2$ numbered by $\{\, 1, n-1, n \,\}$; see Figure \ref{fig:thm:A:a}.
\begin{figure}
\begingroup
\[
  \begin{tikzpicture}[scale=5.5/7]
      \node[name=s, shape=regular polygon, regular polygon sides=9, rotate=-10, minimum size=5.5cm, draw] {}; 
      \draw [thick] (0*360/9:3.5cm) to (2*360/9:3.5cm);
      \draw (2*360/9+12:3.8cm) node { $\cdot$ };
      \draw (2*360/9+20:3.8cm) node { $\cdot$ };
      \draw (2*360/9+28:3.8cm) node { $\cdot$ };      
      \draw (2*360/9:3.90cm) node { $1$ };
      \draw (1*360/9:3.90cm) node { $n$ };
      \draw (0*360/9:4.25cm) node { $n-1$ };
      \draw (0*360/9-12:3.8cm) node { $\cdot$ };
      \draw (0*360/9-20:3.8cm) node { $\cdot$ };
      \draw (0*360/9-28:3.8cm) node { $\cdot$ };                  
      \draw (0cm,0cm) node { $P^1$ };
      \draw (1*360/9:3.1cm) node { $P^2$ };      
  \end{tikzpicture} 
\]
\endgroup
\caption{In the proof of Theorem \ref{thm:A} the dissection $\{\, ( 1,n-1 ), ( n-1,1 ) \,\}$ divides the polygon $P$ into a subpolygon $P^1$ with vertex set $V^1$ numbered by $\{\, 1, \ldots, n-1 \,\}$ and a triangle $P^2$ with vertex set $V^2$ numbered by $\{\, 1, n-1, n \,\}$.}
\label{fig:thm:A:a}
\end{figure}
Set $c^{ \alpha } = c \Big|_{ \diag P^{ \alpha } }$ for $\alpha$ in $\{\, 1,2 \,\}$.  Then
\begin{align*}
  \det M_c
  & \stackrel{ \rm (a) }{ = } \overline{ -1 }
      \cdot \overline{ c_{ 1,n-1 }^{ -1 }c_{ n-1,1 }^{ -1 } }
      \cdot \det( M_{ c^1 } )
      \cdot \det( M_{ c^2 } ) \\[2mm]
  & \stackrel{ \rm (b) }{ = } \overline{ -1 }
      \cdot \overline{ c_{ 1,n-1 }^{ -1 }c_{ n-1,1 }^{ -1 } }
      \cdot \Big( \overline{ -( -2 )^{ n-3 } }
      \cdot \overline{ c_{ 12 }c_{ 23 } \ldots c_{ n-2,n-1 }c_{ n-1,1 } } \Big)
      \cdot \Big( \overline{ 2 }
      \cdot \overline{ c_{ 1,n-1 }c_{ n-1,n }c_{ n1 } }  \Big) \\[2mm]
  & = \overline{ -( -2 )^{ n-2 } }
      \cdot 
      \overline{ c_{ 12 }c_{ 23 } \cdots c_{ n-1,n }c_{ n1 } }
\end{align*}
as desired, where (a) is by Proposition \ref{pro:A} and (b) is by the theorem at values $n-1$ and $3$.
\end{proof}

\begin{proof}
[Proof of Theorem \ref{thm:B} (The non-commutative Laurent phenomenon)]
We start with an observation on the $T$-path formula stated in Theorem \ref{thm:B} and Equation \eqref{equ:def:T-path_formula:a} in Definition \ref{def:T-path_formula}: If the diagonal $( i,k )$ crosses no diagonal in the dissection in the formula, then the formula holds trivially for any map $c$, because the unique $T$-path from $i$ to $k$ with respect to the dissection is $( i,k )$.

To prove that condition (i) in the theorem implies condition (ii), we use induction on $|D|$.

$|D| = 0$:  Let $c : \diag P \xrightarrow{} R$ satisfy condition (i), which in this case is vacuous.  The condition stated after the formula in (ii) is vacuously true.  To show that $c$ satisfies condition (ii), it suffices to note that each diagonal $( i,k )$ in $P$ crosses no diagonal in $D$ and apply the observation at the start of the proof.

$|D| > 0$:  Assume the implication (i)$\Rightarrow$(ii) holds for smaller values of $|D|$.  Let $c : \diag P \xrightarrow{} R$ satisfy condition (i).  The condition stated after the formula in (ii) holds by Definition \ref{def:weak_frieze}(i).  To show that $c$ satisfies condition (ii), it suffices to let $i$ and $k$ be different vertices of $P$ and prove that $c$ satisfies the $T$-path formula from $i$ to $k$ with respect to $D$.

If $( i,k )$ crosses no diagonal in $D$, then we are done by the observation at the start of the proof.

Assume that $( i,k )$ crosses at least one diagonal in $D$.  Repeated application of \cite[lem.\ 3.4]{Canakci-Jorgensen} shows that it is enough to prove that $c$ satisfies the $T$-path formula from $i$ to $k$ with respect to $D$ after the following substitutions illustrated by Figure \ref{fig:thm:B:a}.
\begin{figure}
\begingroup
\[
  \begin{tikzpicture}[scale=1]
      \node[name=s, shape=regular polygon, regular polygon sides=20, rotate=9, minimum size=8cm, draw] {}; 
      \draw [dotted,thick] (11*360/20:4cm) to (1*360/20:4cm);
      \draw (11*360/20:4.3cm) node { $i$ };
      \draw (1*360/20:4.3cm) node { $k$ };
      \draw (8*360/20:4.3cm) node { $t$ };
      \draw (12*360/20:4.3cm) node { $v$ };
      \draw (10*360/20:3.6cm) node { $P^1$ };
      \draw [thick,red] (8*360/20:4cm) to (12*360/20:4cm);
      \draw [thick,red] (8*360/20:4cm) to (14*360/20:4cm);
      \draw [thick,red] (2*360/20:4cm) to (0*360/20:4cm);
      \draw [thick,red] (3*360/20:4cm) to (-2*360/20:4cm);

      \draw [thick] (12*360/20:4cm) to (14*360/20:4cm);
      \draw [thick] (15*360/20:4cm) to (18*360/20:4cm);
      \draw [thick] (4*360/20:4cm) to (8*360/20:4cm);
      \draw [thick] (5*360/20:4cm) to (8*360/20:4cm);
      \draw [thick] (5*360/20:4cm) to (7*360/20:4cm);

      \draw [draw=black, fill=red, opacity=0.1] (8*360/20:4cm) -- (9*360/20:4cm) --(10*360/20:4cm) -- (11*360/20:4cm) -- (12*360/20:4cm) -- cycle;
      \draw [draw=black, fill=red, opacity=0.1] (8*360/20:4cm) -- (12*360/20:4cm) --(14*360/20:4cm) -- cycle;
      \draw [draw=black, fill=red, opacity=0.1] (8*360/20:4cm) -- (14*360/20:4cm) -- (15*360/20:4cm) -- (18*360/20:4cm) -- (3*360/20:4cm) -- (4*360/20:4cm) -- cycle;
      \draw [draw=black, fill=red, opacity=0.1] (2*360/20:4cm) -- (3*360/20:4cm) -- (18*360/20:4cm) -- (19*360/20:4cm) -- (20*360/20:4cm) -- cycle;
      \draw [draw=black, fill=red, opacity=0.1] (0*360/20:4cm) -- (1*360/20:4cm) -- (2*360/20:4cm) -- cycle;

  \end{tikzpicture} 
\]
\endgroup
\caption{This polygon, $P$, has a dissection, $D$, consisting of all the bold diagonals (red and black).  In the proof of Theorem \ref{thm:B}, a diagonal $( i,k )$ is given, and $P$ is replaced by the union of those subpolygons defined by $D$ which $( i,k )$ passes through (pink), while $D$ is replaced by the set of those diagonals in the original $D$ which $( i,k )$ crosses (red).}
\label{fig:thm:B:a}
\end{figure}
\begin{itemize}
\setlength\itemsep{4pt}

  \item  $P$ is replaced by the union of those subpolygons defined by $D$ which $( i,k )$ passes through.

  \item  $c$ is replaced by its restriction to the set of diagonals of the new $P$.

  \item  $D$ is replaced by the set of those diagonals in the original $D$ which $( i,k )$ crosses.  

\end{itemize}
The substitutions establish the situation of Setup \ref{set:P1P2} with $P^1$ being the first subpolygon which $( i,k )$ passes through, $( t,v )$ and $( v,t )$  being the first diagonals in $D$ which $( i,k )$ crosses, and $D^1 = \emptyset$.  Moreover, $i$ is in $U^1$ and $k$ is in $U^2$.

After the substitutions, $c$ still satisfies condition (i).  That is, $c$ is a weak frieze with respect to $D$.  It follows that $c \Big|_{ \diag P^2 }$ satisfies condition (i) with $D^2$ instead of $D$.  Since $D = \{\, ( t,v ), ( v,t ) \,\} \cupdot D^2$ we have $|D^2| < |D|$ so by induction $c \Big|_{ \diag P^2 }$ satisfies condition (ii) with $D^2$ instead of $D$.  That is, $c \Big|_{ \diag P^2 }$ satisfies the $T$-path formula with respect to $D^2$.

This verifies one of the assumptions of Lemma \ref{lem:8}, and the remaining assumptions of the lemma hold because $c$ is a weak frieze with respect to $D$.  So the lemma gives that $c$ satisfies the $T$-path formula from $i$ to $k$ with respect to $D$ as desired.

To prove that condition (ii) in the theorem implies condition (i), we also use induction on $|D|$.

$|D| = 0$:  In this case, condition (i) is vacuously true.

$|D| > 0$:  Assume the implication (ii)$\Rightarrow$(i) holds for smaller values of $|D|$.  Let $c : \diag P \xrightarrow{} R$ satisfy condition (ii).

We can pick diagonals $( t,v )$ and $( v,t )$ in $D$ which place us in the situation of Setup \ref{set:P1P2} with $D^1 = \emptyset$ and the given $D$ equal to $\{\, ( t,v ), ( v,t ) \,\} \cupdot D^2$.  Note that $P^1$ is an ``ear'' of the dissection $D$. 

The restriction $c \Big|_{ \diag P^1 }$ vacuously satifies condition (i) with the dissection $D^1 = \emptyset$ instead of $D$.  It follows from \cite[lem.\ 3.4]{Canakci-Jorgensen} that $c \Big|_{ \diag P^2 }$ satisfies condition (ii) with $D^2$ instead of $D$.  We have $|D^2| < |D|$, so by induction $c \Big|_{ \diag P^2 }$ satisfies condition (i) with $D^2$ instead of $D$.  To sum up, $c \Big|_{ \diag P^{ \alpha } }$ is a weak frieze with respect to $D^{ \alpha }$ for each $\alpha$ in $\{\, 1,2 \,\}$.  Moreover, $c_{ tv }$ and $c_{ vt }$ are in $R^*$ by condition (ii). 

Setting 
\begin{equation}
\label{equ:thm:B:1}
  c^{ \alpha } = c \Big|_{ \diag P^{ \alpha } }
\end{equation}  
for each $\alpha$ in $\{\, 1,2 \,\}$ hence places us in the situation of Lemma \ref{lem:B}, and we will prove that $c$ satisfies condition (i) by showing that $c$ is equal to the weak frieze with respect to $D$ provided by that lemma.

We do so by proving that $c_{ ik }$ is given by the table in Lemma \ref{lem:B}.  In most cases, this holds by Equation \eqref{equ:thm:B:1}.  The only cases not covered are the following.
\begingroup
\renewcommand{\labelenumi}{(\alph{enumi})}
\begin{enumerate}
\setlength\itemsep{4pt}

  \item  $i$ is in $U^1$ and $k$ is in $U^2$.
  
  \item  $i$ is in $U^2$ and $k$ is in $U^1$.

\end{enumerate}
\endgroup

We know that $c$ satisfies condition (ii); in particular, the condition stated after the formula in (ii) ensures that Lemma \ref{lem:computation} can be applied.  We also know that $c \Big|_{ \diag P^2 }$ satisfies condition (ii) with $D^2$ instead of $D$.  Hence, in case (a), Equation \eqref{equ:lem:computation:a} reads
\[
  c_{ ik }
  =
  c_{ it }c_{ vt }^{ -1 }c_{ vk }
  +
  c_{ iv }c_{ tv }^{ -1 }c_{ tk },
\]
and in case (b), Equation \eqref{equ:lem:computation:u} reads the same.  Combining with Equation \eqref{equ:thm:B:1} shows that in case (a), respectively (b), $c_{ ik }$ is indeed given by the top right, respectively bottom left entry in the table in Lemma \ref{lem:B}.
\end{proof}

\begin{proof}
[Proof of Theorem \ref{thm:C} (Gluing weak friezes)]  The general case follows by an easy induction from the case $\kappa = 1$, which holds by Lemma \ref{lem:B}.
\end{proof}

\begin{proof}
[Proof of Theorem \ref{thm:D} (Gluing friezes)]
We will assume $\kappa = 1$ since the general case follows from this by an easy induction.  The dissection \eqref{equ:thm:D:1} in the theorem will be written $\{\, ( t,v ), ( v,t ) \,\}$.  It divides $P$ into subpolygons $P^1$ and $P^2$ on which we have friezes $c^1$ and $c^2$, and the map $c$ from Theorem \ref{thm:C} is the map $c$ from Lemma \ref{lem:B}, applied with $D^1$ and $D^2$ equal to $\emptyset$.

We start with the observation that if, on the other hand, $D^1$ and $D^2$ are arbitrary dissections of $P^1$ and $P^2$, then $c$ is a weak frieze with respect to $\{\, ( t,v ), ( v,t ) \,\} \cupdot D^1 \cupdot D^2$.  This follows from Lemma \ref{lem:B} because $c^1$ and $c^2$ are weak friezes with respect to $D^1$ and $D^2$.  In fact, since $c^1$ and $c^2$ are friezes, they are weak friezes with respect to any dissections of $P^1$ and $P^2$ as remarked in the introduction.

Let $V^1$ and $V^2$ be the vertex sets of $P^1$ and $P^2$.  We use induction on $|V^2|$.

$|V^2| = 3$:  The vertices shared by $P^1$ and $P^2$ are $t$ and $v$.  Denote the third vertex of $P^2$ by $s$.  By symmetry, we can assume $s < t < v$ in the cyclic ordering of the vertices of $P$, hence also in the cyclic ordering of the vertices of $P^2$; see Figure \ref{fig:thm:D:a}.  We must show that $c$ satisfies Definition \ref{def:frieze}.  Note that $c$ has values in $R^*$ by assumption.

$c$ satisfies the triangle relations in Definition \ref{def:frieze}(i): It is enough to show that $c$ satisfies Lemma \ref{lem:frieze}(i'), so we must prove that if $i < j < k$ are pairwise different vertices of $P$, then the triangle relation
\[
  c_{ ij }c_{ kj }^{ -1 }c_{ ki } = c_{ ik }c_{ jk }^{ -1 }c_{ ji }
\]
holds.  This is true by assumption if $i$, $j$, $k$ all sit in $V^1$ or all sit in $V^2$, so assume not.  Then one of $i$, $j$, $k$ is $s$ and the remaining two vertices are in $V^1$ but not both in $V^2$.  By \cite[rmk.\ 2.2(2)]{Cuntz-Holm-Jorgensen-noncomm} we can assume $i = s$, 
so we must prove that if $s < j < k$ are pairwise different vertices of $P$ with $j$, $k$ not both in $V^2$, then
\[
  c_{ sj }c_{ kj }^{ -1 }c_{ ks } = c_{ sk }c_{ jk }^{ -1 }c_{ js }.
\]
By symmetry we can assume $j$ not in $V^2$; in particular, $j$ is not in $\{\, t,v \,\}$.  We have $k \neq t$ since the successor of $s$ is $t$ while $s < j < k$.  We can compute as follows where red factors are replaced in each subsequent line.
\begin{align*}
  &
  { \color{red} c_{ sj } }c_{ kj }^{ -1 }c_{ ks } -
  c_{ sk }c_{ jk }^{ -1 }{ \color{red} c_{ js } } \\[2mm]
  & \;\;\;\; \stackrel{ \rm (a) }{ = }
  ( c_{ st }c_{ kt }^{ -1 }c_{ kj } +
    c_{ sk }c_{ tk }^{ -1 }c_{ tj } )
          c_{ kj }^{ -1 }c_{ ks } -
  c_{ sk }c_{ jk }^{ -1 }
  ( c_{ jk }c_{ tk }^{ -1 }c_{ ts } +
    c_{ jt }c_{ kt }^{ -1 }c_{ ks } ) \\[2mm]
  & \;\;\;\; =
  c_{ st }c_{ kt }^{ -1 }c_{ ks } +
  c_{ sk }{ \color{red} c_{ tk }^{ -1 }c_{ tj }c_{ kj }^{ -1 } }c_{ ks } -
  c_{ sk }c_{ tk }^{ -1 }c_{ ts } -
  c_{ sk }c_{ jk }^{ -1 }c_{ jt }c_{ kt }^{ -1 }c_{ ks } \\[2mm]
  & \;\;\;\; \stackrel{ \rm (b) }{ = }
  c_{ st }c_{ kt }^{ -1 }c_{ ks } +
  c_{ sk }c_{ jk }^{ -1 }c_{ jt }c_{ kt }^{ -1 }c_{ ks } -
  c_{ sk }c_{ tk }^{ -1 }c_{ ts } -
  c_{ sk }c_{ jk }^{ -1 }c_{ jt }c_{ kt }^{ -1 }c_{ ks } \\[2mm]
  & \;\;\;\; =
  c_{ st }c_{ kt }^{ -1 }c_{ ks } -
  c_{ sk }c_{ tk }^{ -1 }c_{ ts } \\[2mm]
  & \;\;\;\; = (*).
\end{align*}
We explain the labelled equalities.  For equality (a), observe that $s < j < k$ and that $s$ and $t$ are consecutive vertices of $P$ while $j \neq t$, so we even have
\begin{equation}
\label{equ:thm:D:u}
  s < t < j < k.
\end{equation}
The observation near the start of the proof with $D^1 = \{\, ( t,k ),( k,t ) \,\}$ and $D^2 = \emptyset$ says that $c$ is a weak frieze with respect to $\{\, ( t,v ), ( v,t ), ( t,k ), ( k,t ) \,\}$ so \eqref{equ:thm:D:u} gives exchange relations which imply (a).  Equality (b) follows from the inverted triangle relation in Lemma \ref{lem:inverted_triangle_relations} for the pairwise different vertices $t$, $j$, $k$ of $P^1$ since $c \Big|_{ \diag P^1 } = c^1$ is a frieze.

If $k = v$ then
\[
  (*) =
  c_{ st }c_{ vt }^{ -1 }c_{ vs } -
  c_{ sv }c_{ tv }^{ -1 }c_{ ts } = 0
\]
where the last equality holds because $c \Big|_{ \diag P^2 } = c^2$ is a frieze.

If $k \neq v$ then we can compute as follows where red factors are replaced in each subsequent line.
\begin{align*}
  (*) & =
  c_{ st }c_{ kt }^{ -1 }{ \color{red} c_{ ks } } -
  { \color{red} c_{ sk } }c_{ tk }^{ -1 }c_{ ts } \\[2mm]
  & \stackrel{ \rm (c) }{ = }
  c_{ st }c_{ kt }^{ -1 }
  ( c_{ kv }c_{ tv }^{ -1 }c_{ ts } + 
    c_{ kt }c_{ vt }^{ -1 }c_{ vs } ) -
  ( c_{ st }c_{ vt }^{ -1 }c_{ vk } +
    c_{ sv }c_{ tv }^{ -1 }c_{ tk } )
  c_{ tk }^{ -1 }c_{ ts } \\[2mm]
  & =
  c_{ st }{ \color{red} c_{ kt }^{ -1 }
  c_{ kv }c_{ tv }^{ -1 } }c_{ ts } + 
  { \color{red} c_{ st }c_{ vt }^{ -1 }c_{ vs } } -
  c_{ st }c_{ vt }^{ -1 }c_{ vk }c_{ tk }^{ -1 }c_{ ts } -
  c_{ sv }c_{ tv }^{ -1 }c_{ ts } \\[2mm]
  & \stackrel{ \rm (d) }{ = }
  c_{ st }c_{ vt }^{ -1 }
  c_{ vk }c_{ tk }^{ -1 }c_{ ts } + 
  c_{ sv }c_{ tv }^{ -1 }c_{ ts } -
  c_{ st }c_{ vt }^{ -1 }c_{ vk }c_{ tk }^{ -1 }c_{ ts } -
  c_{ sv }c_{ tv }^{ -1 }c_{ ts } \\[2mm]
  & = 0.
\end{align*}
We explain the labelled equalities.  For equality (c), observe that $k \neq v$ implies that $k$ is not in $\{\, t,v \,\}$.  Since $s < t < v$, we must have $s < t < k < v$.  This gives exchange relations because $c$ is a weak frieze with respect to $\{\, ( t,v ),( v,t ) \,\}$, and these relations imply (c).  Equality (d) follows from the inverted triangle relation in Lemma \ref{lem:inverted_triangle_relations} for the pairwise different vertices $t$, $v$, $k$ of $P^1$ since $c \Big|_{ \diag P^1 } = c^1$ is a frieze and from the triangle relation for the pairwise different vertices $s$, $t$, $v$ of $P^2$ since $c \Big|_{ \diag P^2 } = c^2$ is a frieze.

$c$ satisfies the exchange relations in Definition \ref{def:frieze}(ii): Since $c \Big|_{ \diag P^1 } = c^1$ is a frieze, $c$ satisfies all exchange relations between diagonals in $P^1$.  The remaining exchange relations involve an internal diagonal not in $P^1$, and such diagonals have $s$ as one of their end points, so we must show that if
\begin{equation}
\label{equ:thm:D:v}
  s < j < k < \ell
\end{equation}  
are pairwise different vertices of $P$, then the corresponding exchange relations are satisfied. 

First, the observation near the start of the proof with $D^1 = \{\, ( j,\ell ),( \ell,j ) \,\}$ and $D^2 = \emptyset$ says that $c$ is a weak frieze with respect to $\{\, ( t,v ), ( v,t ), ( j,\ell ),( \ell,j ) \,\}$ so \eqref{equ:thm:D:v} gives the following exchange relations.
\begin{align}
\label{equ:thm:D:a}
  c_{ sk } 
  & = c_{ sj }c_{ \ell j }^{ -1 }c_{ \ell k }
    + c_{ s\ell }c_{ j\ell }^{ -1 }c_{ jk } \\[2mm]
\nonumber
  c_{ ks }
  & = c_{ k\ell }c_{ j\ell }^{ -1 }c_{ js }
    + c_{ kj }c_{ \ell j }^{ -1 }c_{ \ell s }
\end{align}

Secondly, we prove the exchange relation
\[
  c_{ j\ell } 
  = c_{ jk }c_{ sk }^{ -1 }c_{ s\ell }
  + c_{ js }c_{ ks }^{ -1 }c_{ k\ell }
\]
by computing as follows where red factors are replaced in each subsequent line.
\begin{align*}
  & 
  ( c_{ j\ell }
    - c_{ jk }c_{ sk }^{ -1 }c_{ s\ell }
    - c_{ js }c_{ ks }^{ -1 }c_{ k\ell } )
  c_{ s\ell }^{ -1 }c_{ sk } \\[2mm]
  & \;\;\;\; =
  c_{ j\ell }c_{ s\ell }^{ -1 }{ \color{red} c_{ sk } }
  - c_{ jk }
  - c_{ js }c_{ ks }^{ -1 }c_{ k\ell }c_{ s\ell }^{ -1 }c_{ sk } \\[2mm]
  & \;\;\;\; \stackrel{ \rm (e) }{ = }
  c_{ j\ell }c_{ s\ell }^{ -1 }  
  ( c_{ sj }c_{ \ell j }^{ -1 }c_{ \ell k }
    + c_{ s\ell }c_{ j\ell }^{ -1 }c_{ jk } )
  - c_{ jk }
  - c_{ js }c_{ ks }^{ -1 }c_{ k\ell }c_{ s\ell }^{ -1 }c_{ sk } \\[2mm]
  & \;\;\;\; =
  c_{ j\ell }c_{ s\ell }^{ -1 }c_{ sj }c_{ \ell j }^{ -1 }c_{ \ell k }
  + c_{ jk }
  - c_{ jk }
  - c_{ js }c_{ ks }^{ -1 }c_{ k\ell }c_{ s\ell }^{ -1 }c_{ sk } \\[2mm]
  & \;\;\;\; =
  c_{ j\ell }{ \color{red} c_{ s\ell }^{ -1 }c_{ sj }c_{ \ell j }^{ -1 } }c_{ \ell k }
  - c_{ js }{ \color{red} c_{ ks }^{ -1 }c_{ k\ell }c_{ s\ell }^{ -1 } }c_{ sk } \\[2mm]
  & \;\;\;\; \stackrel{ \rm (f) }{ = }
  c_{ j\ell }c_{ j\ell }^{ -1 }c_{ js }c_{ \ell s }^{ -1 }c_{ \ell k }
  - c_{ js }c_{ \ell s }^{ -1 }c_{ \ell k }c_{ sk }^{ -1 }c_{ sk } \\[2mm]
  & \;\;\;\; =
  c_{ js }c_{ \ell s }^{ -1 }c_{ \ell k }
  - c_{ js }c_{ \ell s }^{ -1 }c_{ \ell k } \\[2mm]
  & \;\;\;\; = 0.
\end{align*}
We explain the labelled equalities.  Equality (e) holds by Equation \eqref{equ:thm:D:a}.  Equality (f) follows from the inverted triangle relation in Lemma \ref{lem:inverted_triangle_relations} for the pairwise different vertices $s$, $j$, $\ell$, respectively $s$, $k$, $\ell$ because all triangle relations were established for $c$ earlier in the proof.

Thirdly, the exchange relation
\[
  c_{ \ell j } 
  = c_{ \ell s }c_{ ks }^{ -1 }c_{ kj }
  + c_{ \ell k }c_{ sk }^{ -1 }c_{ sj }
\]
follows by an analogous computation.

$|V^2| \geqslant 4$:  Recall that we are proving the theorem with $\kappa = 1$ by induction on $| V^2 |$ and that $c$ is glued from $c^1$ and $c^2$ using Theorem \ref{thm:C}.  Assume the theorem with $\kappa = 1$ holds for smaller values of $|V^2|$.

Pick vertices $t'$ and $v'$ of $P^2$ such that $\{\, ( t',v' ), ( t',v' ) \,\}$ is a dissection of $P^2$ into subpolygons $P^{2 \prime }$ and $P^{2 \prime \prime }$ with numbers of vertices satisfying $|V^{2 \prime }| < |V^2|$ and $|V^{2 \prime \prime }| < |V^2|$, and such that $P^{ 2 \prime }$ shares the edges $( t,v )$ and $( v,t )$ with $P^1$; see Figure \ref{fig:thm:D:b}.
\begin{figure}
\begingroup
\[
  \begin{tikzpicture}[scale=5.5/7]
      \node[name=s, shape=regular polygon, regular polygon sides=9, rotate=-10, minimum size=5.5cm, draw] {}; 
      \draw [thick] (0*360/9:3.5cm) to (2*360/9:3.5cm);
      \draw (2*360/9:3.85cm) node { $t$ };
      \draw (1*360/9:3.85cm) node { $s$ };
      \draw (0*360/9:3.8cm) node { $v$ };
      \draw (0cm,0cm) node { $P^1$ };
      \draw (1*360/9:3.1cm) node { $P^2$ };      
  \end{tikzpicture} 
\]
\endgroup
\caption{In the proof of Theorem \ref{thm:D}, in the induction start, the dissection $\{\, ( t,v ), ( v,t ) \,\}$ divides the polygon $P$ into a subpolygon $P^1$ and a triangle $P^2$ whose third vertex is denoted $s$.}
\label{fig:thm:D:a}
\end{figure}
\begin{figure}
\begingroup
\[
  \begin{tikzpicture}[scale=0.7]
      \node[name=s, shape=regular polygon, regular polygon sides=20, rotate=-10, minimum size=4.9cm, draw] {}; 
      \draw [thick] (6*360/20:3.5cm) to (17*360/20:3.5cm);
      \draw [thick] (5*360/20:3.5cm) to (19*360/20:3.5cm);
      \draw (6*360/20:3.8cm) node { $t$ };
      \draw (17*360/20:3.8cm) node { $v$ };
      \draw (5*360/20:3.8cm) node { $t'$ };
      \draw (19*360/20:3.8cm) node { $v'$ };
      \draw (11.5*360/20:1.1cm) node { $P^1$ };
      \draw (2*360/20:1.4cm) node { $P^{ 2 \prime }$ };
      \draw (2*360/20:2.8cm) node { $P^{ 2 \prime \prime }$ };
  \end{tikzpicture} 
\]
\endgroup
\caption{In the proof of Theorem \ref{thm:D}, in the induction step, the dissection $\{\, ( t,v ), ( v,t ) \,\}$ divides the polygon $P$ into subpolygons $P^1$ and $P^2$, and the dissection $\{\, ( t',v' ), ( v',t' ) \,\}$ divides $P^2$ into subpolygons $P^{ 2 \prime }$ and $P^{ 2 \prime \prime }$.}
\label{fig:thm:D:b}
\end{figure}
 We have the following:
\begin{itemize}
\setlength\itemsep{4pt}

  \item  $P^1 \cup P^{ 2 \prime }$ is a polygon divided into subpolygons $P^1$ and $P^{ 2 \prime }$ by the dissection
\[  
  \{\, ( t,v ), ( v,t ) \,\}.
\]
  
  \item  $c^1 : \diag P^1 \xrightarrow{} R^*$ and $c^{ 2 \prime } = c^2 \Big|_{ \diag P^{ 2 \prime } } : \diag P^{ 2 \prime } \xrightarrow{} R^*$ are friezes.
  
  \item  $c^1_{ tv } = c^{ 2 \prime }_{ tv }$ and $c^1_{ vt } = c^{ 2 \prime }_{ vt }$.

\end{itemize}
Theorem \ref{thm:C} can be applied to this data, and it is clear that $c' = c \Big|_{ \diag P^1 \cup P^{ 2 \prime } } : \diag P^1 \cup P^{ 2 \prime } \xrightarrow{} R$ is the resulting unique map which satisfies:
\begin{enumerate}
\setlength\itemsep{4pt}

  \item  $c' \Big|_{ \diag P^1 } = c^1$ and $c' \Big|_{ \diag P^{ 2 \prime } } = c^{ 2 \prime }$.

  \item  $c'$ is a weak frieze with respect to the dissection $\{\, ( t,v ), ( v,t ) \,\}$.

\end{enumerate}
\medskip
We have assumed that $c$ has values in $R^*$, so the same applies to $c'$ whence $c'$ is a frieze by induction because $|V^{2 \prime }| < |V^2|$.

We also have the following:
\begin{itemize}
\setlength\itemsep{4pt}

  \item  $P$ is a polygon divided into subpolygons $P^1 \cup P^{ 2 \prime }$ and $P^{ 2 \prime \prime }$ by the dissection
\[
  \{\, ( t',v' ), ( v',t' ) \,\}.
\]
  
  \item  $c' = c \Big|_{ \diag P^1 \cup P^{ 2 \prime } } : \diag P^1 \cup P^{ 2 \prime } \xrightarrow{} R^*$ and $c^{ 2 \prime \prime } = c^2 \Big|_{ \diag P^{ 2 \prime \prime } } : \diag P^{ 2 \prime \prime } \xrightarrow{} R^*$ are friezes.
  
  \item  $c'_{ t' v' } = c^{ 2 \prime \prime }_{ t' v' }$ and $c'_{ v' t' } = c^{ 2 \prime \prime }_{ v' t' }$.

\end{itemize}
Theorem \ref{thm:C} can be applied to this data, and we claim that $c : \diag P \xrightarrow{} R$ is the resulting unique map which satisfies:
\begingroup
\renewcommand{\labelenumi}{(\roman{enumi}')}
\begin{enumerate}
\setlength\itemsep{4pt}

  \item  $c \Big|_{ \diag P^1 \cup P^{ 2 \prime } } = c'$ and $c \Big|_{ \diag P^{ 2 \prime \prime } } = c^{ 2 \prime \prime }$.

  \item  $c$ is a weak frieze with respect to the dissection $\{\, ( t',v' ), ( v',t' ) \,\}$.

\end{enumerate}
\endgroup
Item (i') is clear and (ii') holds by the observation near the start of the proof with $D^2 = \{\, ( t',v' ), ( v',t' ) \,\}$ by which $c$ is a weak frieze with respect to $\{\, ( t,v ), ( v,t ), ( t',v' ), ( v',t' ) \,\}$.  We have assumed that $c$ has values in $R^*$, whence $c$ is a frieze by induction because $|V^{2 \prime \prime }| < |V^2|$.
\end{proof}

\medskip
\noindent
{\bf Acknowledgement.}
The third author gratefully acknowledges many years of support and hospitality from the Institut f\"ur Algebra, Zahlentheorie und Dis\-krete Mathematik at the Leibniz Universit\"at Hannover.

We thank Karin Baur for comments to a previous version.

This work was supported by a DNRF Chair from the Danish National Research Foundation (grant DNRF156), by a Research Project 2 from the Independent Research Fund Denmark (grant 1026-00050B), and by Aarhus University Research Foundation (grant AUFF-F-2020-7-16).

\end{document}